\documentclass{amsart}

\usepackage{amsmath}
\usepackage{amssymb}
\usepackage{hyperref}

\theoremstyle{plain}
\newtheorem{theorem}{Theorem}[section]
\newtheorem{proposition}[theorem]{Proposition}
\newtheorem{lemma}[theorem]{Lemma}

\theoremstyle{definition}

\newtheorem{definition}[theorem]{Definition}
\newtheorem{remark}[theorem]{Remark}

\DeclareMathOperator{\ord}{ord}
\DeclareMathOperator{\lcm}{lcm}
\DeclareMathOperator{\supp}{supp}

\numberwithin{equation}{section}

\begin{document}

\title{Remarks on a Generalization of the Davenport Constant}

\author{Michael Freeze \and Wolfgang A. Schmid}

\address{The University of North Carolina at Wilmington \\
Department of Mathematics and Statistics \\
Wilmington, North Carolina 28403-2870, USA}

\email{freezem@uncw.edu}

\address{Institut f\"ur Mathematik und Wissenschaftliches Rechnen\\
Karl--Fran\-zens--Universi\-t\"at Graz\\
Heinrichstra{\ss}e 36\\
8010 Graz, Austria}

\email{wolfgang.schmid@uni-graz.at}

\thanks{W.A. Schmid is supported by the Austrian Science Fund (FWF): P18779-N18}



\begin{abstract}
A generalization of the Davenport constant is investigated. For a finite abelian group $G$ and a positive integer $k$, let $\mathsf{D}_k(G)$ denote the smallest $\ell$ such that each sequence over $G$ of length at least $\ell$ has $k$ disjoint non-empty zero-sum subsequences.
For general $G$, expanding on known results, upper and lower bounds on these invariants are investigated and it is proved that the sequence $(\mathsf{D}_k(G))_{k\in\mathbb{N}}$ is eventually an arithmetic progression with difference $\exp(G)$, and several questions arising from this fact are investigated.
For elementary $2$-groups, $\mathsf{D}_k(G)$ is investigated in detail; in particular, the exact values are determined for groups of rank four and five (for rank at most three they were already known).

\end{abstract}

\maketitle

 \section{Introduction}
\label{int}

In this paper we investigate a certain well-established generalization of the Davenport constant:
for a finite abelian group $G$ and a positive integer $k$ let $\mathsf{D}_k(G)$ denote the smallest integer $\ell$ such that each sequence over $G$
of length at least $\ell$ has $k$ disjoint non-empty zero-sum subsequences, i.e., the sum of the elements occurring in the subsequence is the neutral element of the group (also see Section \ref{prel} for a more formal definition).  This is the Davenport constant when $k=1$; we call $\mathsf{D}_k(G)$ the $k$-wise Davenport constant of $G$.

This variant of the Davenport constant was introduced and investigated by F.~Halter-Koch \cite{halterkoch92}, in the context of investigations on the asymptotic behavior of certain counting functions of algebraic integers defined via factorization properties (also see the monograph \cite{geroldingerhalterkochBOOK}, in particular Section 6.1, and the survey article \cite{gaogeroldingersurvey}, in particular Section 5).
Moreover, knowledge of these constants is highly relevant when applying the inductive method to determine or estimate the Davenport constant of certain finite abelian groups. This connection was made explicit by Ch.~Delorme, O.~Ordaz, and D.~Quiroz \cite{delormeetal01} (cf.~Theorem \ref{ub_thm_ind}) and motivated them to investigate these constants; additionally  they introduced other zero-sum invariants, which to a limited extent are also considered in the present paper (see Section \ref{prel} for a definition). Further applications of and results on these invariants can be found in recent papers by G.~Bhowmik, I.~Halupczok, and J.-Ch.~Schlage-Puchta \cite{bhowmikschalge07,bhowmikhallschlage}. For related problems see the recent paper of B.~Girard \cite{girardNEW} and the references there.

The purpose of the present paper is two-fold.
On the one hand, we seek a more detailed understanding of these invariants for general finite abelian groups.
In view of the fact that the Davenport constant is only known for a few special types of groups, we concentrate on establishing upper and lower bounds for these invariants. Moreover, we examine the asymptotic behavior of the sequence $(\mathsf{D}_k(G))_{k \in \mathbb{N}}$.
It is known that the sequence $(\mathsf{D}_k(G) - k \exp(G) )_{k \in \mathbb{N}}$ is bounded (see \cite[Proposition 2.7]{delormeetal01}, also see \cite[Theorem 6.1.5]{geroldingerhalterkochBOOK}).
We show that it is eventually constant, i.e., we show that for each finite abelian group $G$ we have $\mathsf{D}_k(G)= \mathsf{D}_0(G)+ k\exp(G)$ for some $\mathsf{D}_0(G)\in \mathbb{N}_0$ and all sufficiently large $k$.
In fact, it is known that for groups of rank at most two, and for some other types of groups, an equality of the form $\mathsf{D}_k(G)= \mathsf{D}_0(G)+ k\exp(G)$ for some $\mathsf{D}_0(G)\in \mathbb{N}_0$ holds for all $k$.
Yet, it is known that this cannot hold for all finite abelian groups, e.g., it is know that it cannot hold for elementary $2$ and $3$-groups of rank at least $3$ (cf.~\cite{delormeetal01,bhowmikschalge07}).
One of our results, Theorem \ref{elb_thm}, provides a lower bound on $\mathsf{D}_k(G)$ and suggests that indeed this cannot hold for a large variety of groups; very informally, for groups whose rank is large relative to the exponent.

On the other hand, to complement our general results that only yield bounds, we study the $k$-wise Davenport constants of elementary $2$-groups in detail. We obtain the precise value in some cases, for groups of small rank, and obtain refined bounds in the general case.

The organization of the paper is as follows. In Section \ref{prel} we recall some notions and terminology that are relevant for our investigations, and in Section \ref{ub} we recall and expand various results used to obtain upper bounds for the $k$-wise Davenport constants.
In Section \ref{elb} we establish the lower bound mentioned above. Then, in Sections \ref{as} and \ref{rlb}, we focus on the above-mentioned asymptotic result and some problems related to it; in the course of our investigations we also establish an explicit upper bound for an invariant called the successive distance (see Section \ref{prel} for the definition) that is of some independent interest.  Finally, in Section \ref{e2g} we investigate the constants for elementary $2$-groups.

\section{Preliminaries}
\label{prel}

We recall notation and general results used in this paper; our notation and terminology is in line with, e.g., \cite{geroldingerhalterkochBOOK,gaogeroldingersurvey,geroldingerBARC}.

\subsection{General notions}
We denote by $\mathbb{N}$ and $\mathbb{N}_0$ the sets of positive and non-negative integers, respectively.
For $m,n \in \mathbb{Z}$, we denote by $[m,n] = \{ z \in \mathbb{Z} \colon m \le z \le n \}$ the interval of integers.

Let $G$ be a finite abelian group; we use additive notation.
A subset $\{e_1, \dots, e_n\} \subset G\setminus \{0\}$, with $e_i \neq e_j$ for $i\neq j$,
 is called \emph{independent} if $\sum_{i=1}^n m_ie_i =0$, with $m_i\in \mathbb{Z}$, implies that $m_ie_i=0$ for each $i \in [1,n]$. An independent generating subset of $G$ is called a \emph{basis} of $G$. If we say that $\{e_1, \dots, e_n\}$ is a basis of $G$
we implicitly impose that the $e_i$s are pairwise distinct.

For $n \in \mathbb{N}$, we denote by $C_n$ a cyclic group of order $n$.
There exist uniquely determined $1< n_1 \mid \dots \mid n_r$ such that $G \cong C_{n_1} \oplus \dots \oplus C_{n_r}$.
The \emph{exponent} of $G$, denoted $\exp(G)$, is $\lcm(n_1, \dots , n_r)$, i.e.,
$n_r$ for $r\neq 0$ and $1$ for $r=0$; the \emph{rank} of $G$, denoted $\mathsf{r}(G)$, is
$r$. The group $G$ is called a \emph{$p$-group} if $\exp(G)$ is a prime power and \emph{elementary} if $\exp(G)$ is squarefree.
We set $\mathsf{D}^{\ast}(G)=\sum_{i=1}^{r}(n_i - 1) +1$ and denote by $G^-$ an abelian group such that $G\cong G^-\oplus C_{\exp(G)}$.

For subsets $A,B \subset G$ let $A+B=\{a+b\colon a\in A, \, b\in B\}$. The set $A$ is called \emph{sum-free} if for all $a,b,c \in A$ one has
$a+b \neq c$, i.e., $(A+A)\cap A=\emptyset$; and $A$ is called a \emph{Sidon set} if for all $a,b,c,d\in A$ with $|\{a,b,c,d\}|\ge 3$ one has that $a+b \neq c+d$.

\subsection{Sequences}
The central object of this paper are finite sequences over finite abelian groups---we use additive notation for groups---and more specifically conditions on sequences that guarantee the existence of certain subsequences the sum of whose terms is the zero-element of the group.  
In the early works on this subject, from the 1960s, one
actually considered finite sequences in the traditional sense (i.e., repetitions of elements are allowed and the terms are ordered).
However, the fact that the terms are ordered is irrelevant
for the problems under investigation; recall that the underlying
group is abelian and thus the sum of elements is clearly invariant under
reordering.
Indeed, most of the time the ordering is not only irrelevant,
but a cause of technical and notational problems in various arguments;
in particular, in those of the form mainly used in the present paper.

Thus, it is now common to consider, rather than finite sequences
in the traditional sense, a finite collection of elements of $G$ where repetition of elements is allowed yet the elements are not ordered.

One way to formalize this is to consider elements of the free
abelian monoid over $G$.
Another one, which is equivalent except for notation, is to consider finite multi-sets over $G$.
In the present paper, we use the former approach, following the above mentioned works.

Let $(G,+,0)$ be a finite abelian group.
As just mentioned, rather than sequences over $G$ in the traditional sense, in other words
elements of the free non-abelian monoid over $G$, we
consider elements of the free abelian monoid over $G$; still we call
these elements sequences to preserve the intuition and historical context.

We denote by $\mathcal{F}(G)$ the multiplicatively written free abelian monoid
over $G$. By definition, an element $S \in \mathcal{F}(G)$, a \emph{sequence} over $G$, is thus
a (formal) abelian product
\[S= \prod_{g \in G} g^{v_g} \text{ with uniquely determined } v_g \in \mathbb{N}_0;\]
since $G$ is finite, we need no additional assumption on the $v_g$s.  
A possibly more intuitive way of considering $S$ is to note that there exist up to ordering uniquely determined elements $g_1, \dots, g_{\ell} \in G$ such that $S=g_1\dots g_{\ell}$---again, this is a (formal) abelian product---that is,
$S$ corresponds to the sequence (in the traditional sense) $(g_1, \dots, g_{\ell})$ where we merely ``forget'' the ordering of the terms.

The identity element of $\mathcal{F}(G)$ is simply denoted by $1$; we call it the \emph{empty sequence}.
The product of two elements of $\mathcal{F}(G)$, corresponds to the
concatenation of two sequences. 
We say that $T$ is a \emph{subsequence} of $S$ if $T \mid S$ in $\mathcal{F}(G)$;
moreover, if $T \mid S$, then we denote by $T^{-1}S$ the unique element $T' \in \mathcal{F}(G)$
such that $T'T=S$; this corresponds to the subsequence of $S$ of the terms not in $T$.
We call sequences $T_1, \dots, T_k$ \emph{disjoint} subsequences of $S$, 
if their product $T_1 \dots T_k$ is a subsequence of $S$.
Occasionally, we will consider the greatest common divisor of sequences over $G$, being elements 
of a free abelian monoid this is well-defined; this corresponds to, in multi-set terminology, the intersection. Yet, note that disjoint subsequences
do not necessarily have a trivial greatest common divisor.

For $S=g_1\dots g_{\ell}= \prod_{g \in G} g^{v_g}$, we denote by $|S|={\ell} = \sum_{g \in G}v_g$ the \emph{length},
by $\sigma(S)= \sum_{i=1}^{\ell}g_i =\sum_{g \in G}v_g g$ the \emph{sum},
by $\mathsf{v}_g(S)=v_g$ the \emph{ multiplicity} of $g$,
and by $\mathsf{k}(S)=\sum_{i=1}^{\ell} 1/\ord(g_i)=\sum_{g\in G}v_g / \ord(g)$ the \emph{cross number} of $S$.

Finally, we call a sequence \emph{squarefree} if the multiplicity of each element is at most $1$; these are effectively sets, yet for clarity we do \emph{not}
identify squarefree sequences and sets. 

We also consider a subsets of $G$ attached to a sequence over $G$, namely 
we denote by $\supp(S)= \{g_1, \dots, g_n\}= \{g \in G \colon v_g>0\}$ the \emph{support} of $S$.

A sequence $S$ over $G$ is called a \emph{zero-sum sequence} if $\sigma(S)=0$.
Let $\mathcal{B}(G)= \{S\in \mathcal{F}(G) \colon \sigma(S)=0 \}$ the set of
all zero-sum sequences over $G$; as the product of two zero-sum sequences is clearly
again a zero-sum sequence, and the empty sequence is a zero-sum sequence, this is in fact a submonoid
of $\mathcal{F}(G)$.
The letter $\mathcal{B}$ is used, since W.~Narkiewicz~\cite{narkiewicz79} called this
structure block monoid.

A zero-sum sequences is called a \emph{minimal zero-sum sequence} if it is non-empty and each proper and non-empty subsequence is not a zero-sum sequence.
In other words, a non-empty zero-sum sequence $B$ is a minimal zero-sum sequence
if and only if $B$ can not be decomposed into two non-empty zero-sum sequences.
That is, the minimal zero-sum sequences are the irreducible elements, or atoms, of the monoid
$\mathcal{B}(G)$; thus, we denote the set of all minimal zero-sum sequences over $G$ by
$\mathcal{A}(G)$.

Each map $f:G\to G'$ of finite abelian groups, can be extended in 
a unique way to a monoid homomorphism from $\mathcal{F}(G)$ to $\mathcal{F}(G')$ that we also denote by $f$.

\subsection{Factorizations}

In the course of our arguments, we also need
to consider decompositions of zero-sum sequences into minimal zero-sum sequences,
in other words factorizations into irreducible elements in the monoid of zero-sum sequences $\mathcal{B}(G)$.
In particular, to prove one of our results, we need to 
establish an explicit upper bound on the successive distance
of the monoid of zero-sum sequences over a finite abelian group (cf.~below for
a definition);
this invariant was introduced and investigated in Non-unique Factorization Theory. 
On the one hand the definition of this invariant is somewhat involved 
and---while natural in the context of Non-unique Factorization Theory---might 
seem rather artificial as a ``pure'' zero-sum problem, 
and on the other hand this result is of relevance in that context as well.
Thus, we treat it using the notions and notation
typically used in Non-unique Factorization Theory,
and recall notions on factorizations as used in this subject. 
To be consistent throughout the paper and to avoid introducing ad-hoc notation,
we use it in other arguments, too.  

Let $B \in \mathcal{B}(G)$ be a zero-sum sequence.
By the definition of a minimal zero-sum sequence, it follows directly that
$B$ can be factored (or, decomposed) into minimal zero-sum sequences.
To be more precise, there exists some $\ell \in \mathbb{N}_0$ and minimal zero-sum sequences $A_1, \dots, A_{\ell} \in \mathcal{A}(G)$ such that $B = A_1 \dots A_{\ell}$; if $B$ is the empty sequence, than $\ell = 0$.
However, in general neither the integer $\ell$ nor the minimal zero-sum sequences are uniquely
determined by $B$. In various of our arguments, the specific form of
a factorization into minimal zero-sum sequences and relations among distinct factorizations of the same zero-some sequences are relevant. 
We briefly recall some key notions.

One could consider a factorizations of a zero-sum sequences $B\in \mathcal{B}(G)$ as
a sequence (in the traditional sense) $(A_1, \dots, A_{\ell})$ of minimal zero-sum sequences $A_i$ such that  $B=A_1 \dots A_{\ell}$.
For essentially the same reason as for sequences over $G$,
it is advantageous and common to disregard the ordering,
and to consider \emph{formal} abelian products of minimal zero-sum sequences instead.

Thus, we denote by $\mathsf{Z}(G)$ the free abelian monoid over $\mathcal{A}(G)$, the \emph{factorization monoid} over $G$,
i.e., all finite formal abelian products of minimal zero-sum sequences over $G$; the letter $\mathsf{Z}$, traditionally used to denote this structure, is derived from the German word $\emph{Zerlegung}$.
An element $\zeta \in \mathsf{Z}(G)$ is thus a \emph{formal} abelian product 
$\zeta = \prod_{A\in \mathcal{A}(G)}A^{v_A} = A_1 \cdot \ldots \cdot A_{\ell}$ with $v_A\in \mathbb{N}_0$ and $A_i \in \mathcal{A}(G)$.
Here, we use dots in the latter product 
to highlight that it is a formal product; later on, we only do so in critical cases. 
Yet, we exclusively use lower-case Greek letters to denote factorizations, 
while we use upper-case Roman letters to denote sequences, to minimize the risk of confusion.

We denote by $\pi: \mathsf{Z}(G) \to \mathcal{B}(G)$ the map
induced by evaluating the formal product, i.e., $A_1 \cdot \ldots \cdot A_{\ell}\mapsto A_1  \ldots A_{\ell}$.

For each $B \in \mathcal{B}(G)$, we thus have that $\pi^{-1}(B)\subset \mathsf{Z}(G)$
denotes the set of all factorizations of $B$ into minimal zero-sum sequences; this set is denoted by $\mathsf{Z}(B)$.

For a factorization $\zeta = \prod_{A\in \mathcal{A}(G)}A^{v_A} = A_1 \cdot \ldots \cdot A_{\ell}$,
we denote by  $|\zeta| = \sum_{A\in \mathcal{A}(G)} v_A= \ell$
its \emph{length}, i.e., the number of minimal zero-sum sequences
in the decomposition, taking multiplicity into account.
Additionally, we use for $A\in \mathcal{A}(G)$ the notation $\mathsf{v}_A(\zeta)=v_A$ to denote the multiplicity with which the minimal zero-sum sequence
appears in the factorization $\zeta$.
Moreover, for $B\in \mathcal{B}(G)$ a zero-sum sequence, 
 we call $\mathsf{L}(B)= \{|\zeta| \colon \zeta \in \mathsf{Z}(B)\}$
the \emph{set of lengths} of $B$, i.e., the set of all integers $\ell$ such
that $B$ can be factored into $\ell$ (not necessarily distinct) minimal zero-sum sequences.

Since $\mathsf{Z}(G)$ is a free abelian monoid, the notions
``divides'' and ``greatest common divisor'' make sense.
More explicitly, for $\xi, \zeta \in \mathsf{Z}(G)$,
we write $\xi \mid \zeta$, if $\mathsf{v}_A(\xi)\le \mathsf{v}_A(\zeta)$
for each $A\in \mathcal{A}(G)$, and in this case we use the notation $\xi^{-1}\zeta$ to denote the unique $\xi' \in \mathsf{Z}(G)$ such that 
$\xi\xi'=\zeta$; an we write $\gcd(\xi, \zeta)$ to denote 
the element of $\xi' \in \mathsf{Z}(G)$ with $\mathsf{v}_A(\xi')= \min \{\mathsf{v}_A(\xi),\mathsf{v}_A(\zeta)\}$ for each $A\in \mathcal{A}(G)$.

For $\zeta, \xi \in \mathsf{Z}(G)$ let $\mathsf{d}(\zeta,\xi)= \max\{|\gcd(\zeta,\xi)^{-1}\zeta|,|\gcd(\zeta,\xi)^{-1}\xi|\}$ the \emph{distance} of $\zeta$ and $\xi$; 
more informally and intuitively, the distance of two factorizations
is determined by canceling common factors and then considering the maximum 
of the lengths of the parts remaining after this cancellation. Via the distance a metric is defined on $\mathsf{Z}(G)$.

Next, we recall the definition of $\Delta(G)$, the \emph{set of distances} of $\mathcal{B}(G)$, and of $\delta(G)$, the \emph{successive distance} of $\mathcal{B}(G)$ (introduced in \cite{geroldinger88} and \cite{foroutan06}, resp).

For $L = \{\ell_1, \ell_2, \dots\}\subset \mathbb{N}_0$ with $\ell_i < \ell_{i+1}$ let $\Delta(L)= \{\ell_2-\ell_1, \ell_3 - \ell_2, \dots\}$.
For $G$ a finite abelian group, let $\Delta(G)= \bigcup_{B \in \mathcal{B}(G)}\Delta(\mathsf{L}(B))$. It is well-known that
$\Delta(G)\subset [1, \mathsf{D}(G)-2]$, in particular $\Delta(G)=\emptyset $ if $|G| \le 2$.

For $B \in \mathcal{B}(G)$, two distinct elements $k, \ell \in \mathsf{L}(B)$ are called \emph{adjacent lengths} of $B$ if $[\min\{k, \ell\},\max\{k, \ell\}] \cap \mathsf{L}(B)=\{k,\ell\}$. Note that if $k$ and $\ell$ are adjacent lengths, then $|k-\ell| \in \Delta(\mathsf{L}(B))$.
For $\zeta \in \mathsf{Z}(G)$, let $\delta(\zeta)$ denote the smallest $m \in \mathbb{N}_0$ with the following property:  
If $k\in \mathbb{N}$ and $k$ and $|\zeta|$ are adjacent lengths of $\pi(\zeta)$, then there exists some $\xi \in \mathsf{Z}(\pi(\zeta))$ with $|\xi|=k$ and
$\mathsf{d}(\xi,\zeta)\le m$.
And, \[\delta(G)=\sup\{\delta(\zeta)\colon \zeta \in \mathsf{Z}(G)\}.\]
It is known that $\delta(G)$ is finite (see \cite[Theorem 3.9]{foroutan06} and also see \cite[Theorem 3.1.4]{geroldingerhalterkochBOOK}).

Since in this paper we have to consider elements of $G$, sequences over $G$, and elements of the factorization monoid over $G$ simultaneously
we adopt the following notational convention, already used above, to avoid confusion (in critical cases we add an explicit explanation): elements of $G$ are denote by lowercase Latin letters, sequences by uppercase Latin letters, and elements of the factorization monoid by lowercase Greek letters (the last is non-standard).

\subsection{The invariants}

We give a more formal definition of the invariants that are at the center of our interest, the \emph{$k$-wise Davenport constants}.

\begin{definition}
Let $G$ be a finite abelian group.
Let $k\in \mathbb{N}$.  We denote by $\mathsf{D}_k(G)$ the smallest integer $\ell \in \mathbb{N}$ such that every sequence $S \in \mathcal{F}(G)$ of length $|S| \ge \ell$ is divisible by a product of $k$ non-empty zero-sum sequences.
\end{definition}
Note that $\mathsf{D}_1(G) = \mathsf{D}(G)$, where $\mathsf{D}(G)$ is the classical Davenport constant.
As for the Davenport constant,  the constants $\mathsf{D}_k(G)$ can alternatively be defined as the maximum length of certain zero-sum sequences.
Since we make frequent use of this characterization, we recall it and relevant related notions.

We denote by $\mathcal{A}_k(G) = \{B \in \mathcal{B}(G)\colon \max \mathsf{L}(B)=k\}$ and by $\mathcal{M}_k(G)= \{B \in \mathcal{B}(G)\colon \max \mathsf{L}(B)\le k\}$. Note that $\mathcal{A}_1(G) = \mathcal{A}(G)$.

Then, \[\mathsf{D}_k(G) = \max\{|B|\colon B \in \mathcal{M}_k(G)\} = \max\{|B|\colon B \in \mathcal{A}_k(G)\}.\]
The characterization involving $\mathcal{M}_k(G)$ is more classical, indeed it is the form in which $\mathsf{D}_k(G)$ was introduced initially (see \cite[Proposition 1]{halterkoch92}). We use both $\mathcal{A}_k(G)$ and $\mathcal{M}_k(G)$, since this extra flexibility can be useful.
To see that it does not make any difference for the maximal length whether one considers the former or the latter it suffices to note that
$\max\mathsf{L}(B0)=\max\mathsf{L}(B)+1$ for each $B \in \mathcal{B}(G)$, indeed $\mathsf{L}(B0)=\{1+\ell \colon \ell \in \mathsf{L}(B)\}$.
This reasoning additionally shows that $\mathsf{D}_{k+1}(G)\ge \mathsf{D}_k(G)+1$ and in fact each $B \in \mathcal{M}_k(G)$ with $|B|=\mathsf{D}_k(G)$ is an element of $\mathcal{A}_k(G)$.

Our investigations also involve another type of zero-sum invariants. 
For a subset $I \subset \mathbb{N}$, we denote by $\mathsf{s}_I(G)$ the smallest element $\ell \in \mathbb{N} \cup \{\infty\}$ such that each sequence $S \in \mathcal{F}(G)$ of length $|S| \ge \ell$ is divisible by a zero-sum sequence of length in $I$.

Here, we consider these invariants only for $I=[1,k]$ for $k\in \mathbb{N}$,
and use the short-hand notation $\mathsf{s}_{\le k}(G)$ for $\mathsf{s}_{[1,k]}(G)$. 
These constant were introduced in \cite{delormeetal01}, using the notation 
$\mathsf{D}^k(G)$. 
The special case $k = \exp(G)$ is classical and was introduced in \cite{olson69_2} and \cite{vanemdeboas69}. We recall that it is common to denote $\mathsf{s}_{\le \exp(G)}(G)$ by $\eta(G)$ (see \cite{gaogeroldingersurvey} for results on this invariant). Clearly $\mathsf{s}_{\le k+1}(G)\le \mathsf{s}_{\le k}(G)$ for each $k \in \mathbb{N}$. Moreover, as shown in \cite{delormeetal01}, $\mathsf{s}_{\le k}(G)= \mathsf{D}(G)$ for each $k \ge \mathsf{D}(G)$ and $\mathsf{s}_{\le k}(G)= \infty$ for $k < \exp(G)$.

Finally, we recall some results on $\mathsf{D}(G)$ and $\eta(G)$. It is well-known that $\mathsf{D}(G) \ge \mathsf{D}^{\ast}(G)$ and that in case $G$ is a $p$-group or $\mathsf{r}(G)\le 2$ equality holds (see, e.g., \cite[Section 5]{geroldingerhalterkochBOOK}). Moreover, it is known that $\mathsf{D}(G) \le \eta(G) \le |G|$ (see \cite{gaoyang97} and also see \cite[Theorem 4.2.7]{geroldingerBARC}).

\section{Upper bounds}
\label{ub}

In this section we state several results that can be used to derive upper bounds for $k$-wise Davenport constants.
These results build on well-known methods used in earlier investigations on this problem, which are mentioned in Section \ref{int}. Mainly, we dissect and slightly expand these results, to make them more directly applicable in the investigations of the following sections.

In the following proposition we collect some basic facts relating $\mathsf{D}_k(G)$ and $\mathsf{D}_{k+1}(G)$; for closely related results see \cite{halterkoch92,delormeetal01,geroldingerhalterkochBOOK}.
\begin{proposition}
\label{ub_prop_bas}
Let $G$ be a finite abelian group and let $k \in \mathbb{N}$.
\begin{enumerate}
\item Let $B$ be a zero-sum sequence over $G$ with $\max\mathsf{L}(B)=k+1$ and let $U$ be a minimal zero-sum sequence over $G$ with $U \mid B$.
Then, $\max\mathsf{L}(U^{-1}B)\le k$. Moreover, $\max\mathsf{L}(U^{-1}B) = k$ if and only if there exists some factorization $\zeta$ of $B$ with length $k+1$ such that $U \mid \zeta$.
\item Let $M = \min \{ |U| \colon U\in \mathcal{A}(G), \, U \mid B \text{ for } B \in \mathcal{A}_{k+1}(G), \, |B|= \mathsf{D}_{k+1}(G)\}$,
i.e., $M$ is the minimum of the lengths of all minimal zero-sum sequences
dividing some zero-sum sequence $B$ over $G$ with $\max \mathsf{L}(B)=k+1$ and
(maximal) length $|B|= \mathsf{D}_{k+1}(G)$. 
Then, $\mathsf{D}_{k+1}(G)\le \mathsf{D}_k(G) + M$.
\item For each $\ell \in \mathbb{N}$,  $\mathsf{D}_{k+1}(G) \le \max \{\mathsf{D}_k(G) + \ell, \mathsf{s}_{\le \ell}(G)-1\}$.
In particular, if $\mathsf{D}_k(G) \ge \eta(G) - 1 -\exp(G)$, then $\mathsf{D}_{k+1}(G)\le \mathsf{D}_k(G) +\exp(G)$.
\end{enumerate}
\end{proposition}
\begin{proof}
1. We observe that there is a one-to-one correspondence
between the set of all factorizations of $B$ that contain $U$ 
and the set of all factorizations of the zero-sum sequence $U^{-1}B$ (given by removing one $U$ from the factorization). 
This implies the claim. \\
2. Let $U$ be a minimal zero-sum sequence over $G$ and let $B$ be a zero-sum sequence over $B$ with $\max \mathsf{L}(B)=k+1$, $|B|= \mathsf{D}_{k+1}(G)$, and $U \mid B$ such that $|U|=M$.
By 1.\ we get $\max\mathsf{L}(U^{-1}B)\le k$. Thus, $|U^{-1}B| \le \mathsf{D}_k(G)$.
This implies that $\mathsf{D}_{k+1}(G)= |B| = |U^{-1}B| + |U| \le   \mathsf{D}_k(G) + M$.\\
3. Let $B \in \mathcal{A}_{k+1}(G)$ with $|B| = \mathsf{D}_{k+1}(G)$. Suppose $|B|> \max \{\mathsf{D}_k(G) + \ell, \mathsf{s}_{\le \ell}(G)-1\}$.
By definition of $\mathsf{s}_{\le \ell}(G)$ we know that there exists some $U$
minimal zero-sum sequence with  $U \mid B$  and $|U|\le \ell$.
By 2., observe that $M$ is at most $\ell$, we get $\mathsf{D}_{k+1}(G)\le \mathsf{D}_k(G) + \ell$, a contradiction.
The ``in particular''-statement is merely the special case $\ell  = \exp(G)$.
\end{proof}

Moreover, we make frequent use of the following result, which allows us to derive upper bounds on $\mathsf{D}_k(G)$ in terms of the constants $\mathsf{s}_{\le k}(G)$; it expands on well-known results (cf. Remark \ref{ub_rem}).

\begin{proposition}
\label{ub_prop_lengthlb}
Let $G$ be a finite abelian group. Let $n \in \mathbb{N}_0$ and let ${\overline{\ell}}=(\ell_1,\dots,  \ell_n) \in \mathbb{N}^n$ with $\ell_{i}< \ell_{i+1}$ for each $i \in [1, n-1]$.
For $m \in \mathbb{N}_0$, and $i \in [1,n]$, we recursively define
\[k_{i}^{\overline{\ell}}(m) = \max\left\{0 , \left \lceil \frac{ m - ( \sum_{j=1}^{i-1} k_j^{\overline{\ell}}(m)\ell_j ) - \mathsf{s}_{\le \ell_{i}}(G) + 1}{\ell_i}\right  \rceil \right\}\]
and
\[k_{n+1}^{\overline{\ell}}(m)= \max \left\{0, \left\lceil \frac{ m - \sum_{j=1}^{n} k_j^{\overline{\ell}}(m)\ell_j }{\mathsf{D}(G)}\right \rceil \right\}.\]
\begin{enumerate}
\item For $B$ a zero-sum sequence over $G$, we have $\max \mathsf{L}(B) \ge \sum_{i=1}^{n+1} k_i^{\overline{\ell}}(|B|)$.
\item Let $k \in \mathbb{N}$  and let $M\in \mathbb{N}$ be maximal such that $\sum_{i=1}^{n+1} k_i^{\overline{\ell}}(M) \le k$.
Then, $\mathsf{D}_{k}(G)\le M$.
\end{enumerate}
\end{proposition}
\begin{proof}
We note that $\sum_{i=1}^{n+1} k_i^{\overline{\ell}}(\cdot)$ is a non-decreasing function; yet, note that this is not true for each summand individually.

\noindent
1. Let $B$ be a zero-sum sequence over $G$.
We induct on $n$.
Suppose $n=0$. We have to show that $\max \mathsf{L}(B) \ge \lceil |B| / \mathsf{D}(G)\rceil$. Let $B=A_1\dots A_{k}$ with minimal zero-sum sequences $A_i \in \mathcal{A}(G)$.
We have $|B|= \sum_{i=1}^k |A_i| \le k \mathsf{D}(G)$. Thus, $k \ge  |B| / \mathsf{D}(G)$ and the claim is established.

Let $n \ge 1$.
First suppose $|B|< \mathsf{s}_{\le \ell_{1}}(G)$. Then $k_1^{\overline{\ell}}(|B|)=0$.
By induction hypothesis, applied to ${\overline{\ell'}}=(\ell_2,  \dots, \ell_{n})$ we get that
$\max \mathsf{L}(B) \ge \sum_{i=1}^{n} k_i^{\overline{\ell'}}(|B|)$.
Noting that $k_i^{\overline{\ell'}}(|B|) = k_{i+1}^{\overline{\ell}}(|B|)$ for each $i \in [1,n]$, we have
$\max \mathsf{L}(B) \ge \sum_{i=2}^{n+1} k_i^{\overline{\ell}}(|B|)=\sum_{i=1}^{n+1} k_i^{\overline{\ell}}(|B|)$.

Second suppose $|B|\ge \mathsf{s}_{\le \ell_{1}}(G)$. Thus we have $|B|\ge ( k_1^{\overline{\ell}}(|B|) - 1 ) \ell_1 + \mathsf{s}_{\le \ell_1}(G)$, and by the definition of $\mathsf{s}_{\le \ell_{1}}(G)$ it follows that there exist $U_1, \dots, U_{k_1^{\overline{\ell}}(|B|)}\in \mathcal{A}(G)$ with $|U_i|\le \ell_1$ such that $U_1 \dots U_{k_1^{\overline{\ell}}(|B|)}\mid B$. We set $B'= (U_1 \dots U_{k_1^{\overline{\ell}}(|B|)})^{-1}B$ and note that $|B'|\ge |B|- k_1^{\overline{\ell}}(|B|)\ell_1 $.
By induction hypothesis, applied to $\overline{\ell'}=(\ell_2,  \dots, \ell_{n})$ and $B'$,
we get that
\[
\max \mathsf{L}(B')\ge \sum_{i=1}^{n} k_i^{\overline{\ell'}}(|B|')\ge \sum_{i=1}^{n} k_i^{\overline{\ell'}}(|B| - k_1^{\overline{\ell}}(|B|)\ell_1)= \sum_{i=2}^{n+1} k_i^{\overline{\ell}}(|B|);
\]
where the second inequality holds, since  $\sum_{i=1}^{n} k_i^{\overline{\ell}}(\cdot)$ is non-decreasing.
Since we know that $\max \mathsf{L}(B) \ge \max \mathsf{L}(B') +  k_1^{\overline{\ell}}(|B|)$, the claim follows.

\noindent
2. Let $B$ be a zero-sum sequence over $G$ such that $|B| > M$. Then $\sum_{i=1}^{n+1} k_i^{\overline{\ell}}(|B|) > k$.
Thus, by 1.\ we have $\max \mathsf{L}(B) > k$, and the claim follows.
\end{proof}
In view of the results on $\mathsf{s}_{\le k}(G)$ recalled in Section \ref{prel}, it does not make much sense to consider this lemma other than for $\exp(G)\le \ell_1 < \dots< \ell_n \le  \mathsf{D}(G)$.
Indeed, also to consider $\ell_n = \mathsf{D}(G)$  only adds redundancy.

We point out the following special cases contained in this result (see \cite[Proposition 1]{halterkoch92}, \cite[Lemma 2.4]{delormeetal01} and \cite[Lemma 6.1.3]{geroldingerhalterkochBOOK}).

\begin{remark}
\label{ub_rem}
Let $G$ be a finite abelian group.
\begin{enumerate}
\item Let $k \in \mathbb{N}$. By Proposition \ref{ub_prop_lengthlb}.2 with $n=0$, we get $\mathsf{D}_k(G)\le k \mathsf{D}(G)$.
\item Let $B$ be a zero-sum sequence over $G$. By Proposition \ref{ub_prop_lengthlb}.1 with $n=1$ and $\ell_1 \in [\exp(G), \mathsf{D}(G)-1]$, we get
$\max \mathsf{L}(B)\ge (|B|-\mathsf{s}_{\le \ell_1}(G)+1)/\ell_1$; in particular $\max \mathsf{L}(B)\ge (|B|-\eta(G)+1)/\exp(G)$.
\item Let $k \in \mathbb{N}$. By Proposition \ref{ub_prop_lengthlb}.2 with $n=1$ and $\ell_1 \in [\exp(G), \mathsf{D}(G)-1]$, we get $\mathsf{D}_k(G)\le (k-1) \ell_1 + \max \{\mathsf{D}(G), \mathsf{s}_{\le \ell_1}(G) - \ell_1\}$; in particular $\mathsf{D}_k(G)\le (k-1) \exp(G)+ \max \{\mathsf{D}(G), \eta(G)-\exp(G)\}$.
\end{enumerate}
\end{remark}

An obstacle impeding the application of Proposition \ref{ub_prop_lengthlb}, in its general form, is the fact that, for general $G$,
little is known about $\mathsf{s}_{\le \ell}(G)$ for $\exp(G)< \ell < \mathsf{D}(G)$.
For precise results in special cases see \cite{balasbhowmik07,bhowmikschalge07,delormeetal01}.
In Section \ref{e2g} we obtain some information on these invariants for elementary $2$-groups, which we then use in our investigations of $k$-wise Davenport constants of elementary $2$-groups.
Below, we summarize information that can be obtained on these constants by direct application of a method first used in \cite{olson69_2,vanemdeboas69} to determine $\eta(C_p \oplus C_p)$.

\begin{lemma}
\label{ub_lem_D^k}
Let $G$ be a finite abelian group. Let $n \in \mathbb{N}$.
Let $m= \max \{ \lfloor \lfloor \mathsf{D}(G\oplus C_n)/n \rfloor n/2 \rfloor,  \lfloor \mathsf{D}(G)/n \rfloor n \}$.
Then, $\mathsf{s}_{\le m} ( G )\le \mathsf{D} ( G \oplus C_n )$.
\end{lemma}
\begin{proof}
Let $S$ be a sequence over $G$ with $|S| \ge \mathsf{D} ( G \oplus C_n )$. Let $e\in G \oplus C_n $ such that
$G \oplus C_n = G + \langle e \rangle$.
We consider the map $\iota: G \to G\oplus C_n $ defined via $\iota(g)= g+e$.
Since $|\iota(S)|\ge \mathsf{D} ( G \oplus C_n )$, there exists a non-empty subsequence $B \mid S$ with $|B| \le \mathsf{D}(G\oplus C_n)$ such that $\sigma(\iota(B))=0$.
We have $n \mid |B|$, thus  $|B| \le \lfloor \mathsf{D}(G\oplus C_n)/n \rfloor n$, and $\sigma(B)=0$. If $|B|\le \lfloor \mathsf{D}(G)/n \rfloor n $, we are done. Thus, assume this is not the case. Then $|B| > \mathsf{D}(G)$ and $B$
is not a minimal zero-sum sequence. 
Consequently $B=B_1B_2$ with non-empty zero-sum sequences $B_1$ and $B_2$.  At least one of these two sequences has length at most $|B|/2$, and the claim follows.
\end{proof}

Of course, to apply this result in concrete situations, one needs knowledge of the size of $\mathsf{D} ( G \oplus C_n )$.
For example, this result can be applied for $p$-groups and $n$ a power of $p$ or cyclic groups.

For illustration, we give the following result for elementary $p$-groups (cf.~\cite{delormeetal01} for related results).

\begin{proposition}
Let $p$ be a prime and $r \in \mathbb{N}_{\ge 2}$. Let  $k \in \mathbb{N}$ and let $m \in \mathbb{N}$ such that $r(p-1)+1< 2p^m$.
Then
\[
\mathsf{D}_k(C_p^r)\le \min \{(k(r-1)+1)p - r+1, (k-1)p^m + r(p-1)+1\}.
\]
\end{proposition}
\begin{proof}
We show that both $(k(r-1)+1)p - r+1$ and $(k-1)p^m + r(p-1)+1$ are
upper bounds for $\mathsf{D}_k(C_p^r)$.

First, applying Lemma \ref{ub_lem_D^k} with $n=p$, we get that $\mathsf{s}_{\le (r-1)p}(C_p^r)\le (r+1)(p-1)+1$.
Thus, by Remark \ref{ub_rem}.3 with $\ell_1= (r-1)p$ we have $\mathsf{D}_k(C_p^r)\le (k-1)(r-1)p+ r(p-1)+1= (k(r-1)+1)p - r+1$.

Second, applying Lemma \ref{ub_lem_D^k} with $n=p^m$, we get that $\mathsf{s}_{\le p^m}(C_p^r)\le r(p-1) + (p^m - 1) + 1$.
Thus, by Remark \ref{ub_rem}.3 with $\ell_1= p^m$ we have $\mathsf{D}_k(C_p^r)\le (k-1)p^m+ r(p-1)+1$.
\end{proof}
We point out that it is known that for $r=2$ the former and the latter, with $m=1$, bound is sharp.
A more general class of groups for which this method works well are $p$-groups where the exponent is ``large'' relative to the order of the group.
Moreover, for cyclic groups the bound obtained via this method is sharp as well. We refer to Remark \ref{as_rem} for details.

As indicated above, this is not a novel method of proving these results, only a different way of phrasing the original arguments. We include it only to illustrate that there seems to be no loss of precision when dissecting the arguments in the way we did.

Another way to obtain upper bounds on $\mathsf{D}_k(G)$ is the inductive method (see, e.g., \cite[Section 5.8]{geroldingerhalterkochBOOK}).
We recall the following result. The general statement is due to \cite{delormeetal01}; for the ``in particular''-statement in case $\ell = \exp(G/G')$ see also \cite[Lemma 6.1.3]{geroldingerhalterkochBOOK}. Note that, for $k=1$, this result encodes certain applications of the inductive method in the investigation of the Davenport constant, e.g., the classical argument used to determine the Davenport constant for groups of rank $2$.

\begin{theorem}
\label{ub_thm_ind}
Let $G$ be a finite abelian group and $G'\subset G$ a subgroup. Let $k , \ell \in \mathbb{N}$.
Then $\mathsf{D}_k(G) \le \mathsf{D}_{\mathsf{D}_k(G')}(G/G')$.
In particular, $\mathsf{D}_k(G) \le (\mathsf{D}_k(G')-1)\ell  +  \max \{\mathsf{D}(G/G'), \mathsf{s}_{\le \ell}(G/G') - \ell\}$.
\end{theorem}
\begin{proof}
The inequality $\mathsf{D}_k(G) \le \mathsf{D}_{\mathsf{D}_k(G')}(G/G')$ is proved in \cite[Proposition 2.6]{delormeetal01} (in a more general context).
Using the bound given in Remark \ref{ub_rem}.3 the ``in particular'' statement follows.
\end{proof}
In view of the proof of this result, it is apparent that the ``in particular'' statement can be strengthened by using stronger results on the generalized Davenport constants of $G/G'$.
In Section \ref{e2g} we give, for elementary $2$-groups, another bound of this type, yet exploiting the special structure of these groups.

\section{Explicit lower bounds}
\label{elb}

In this section we establish and recall explicit lower bounds for $\mathsf{D}_k(G)$.
By explicit we mean that the bounds for $\mathsf{D}_k(G)$ do not depend on $\mathsf{D}_{k'}(G)$ for $k'< k$, opposed to the ones that we obtain in Section \ref{rlb}.

First we recall some results regarding lower bounds for $\mathsf{D}_k(G)$ (see \cite{halterkoch92} and \cite[Section 6.1]{geroldingerhalterkochBOOK}).
Let $G=G_1\oplus G_2$ and let $k_1, k_2 \in \mathbb{N}$. Then
\begin{equation}
\label{elb_eq_1}
\mathsf{D}_{k_1+k_2-1}(G)\ge (\mathsf{D}_{k_1}(G_1)-1) + (\mathsf{D}_{k_2}(G_2)-1) +1.
\end{equation}
Using the well-known fact that $\mathsf{D}_{k}(C_n)\ge kn$, obtainable by considering the sequence $g^{kn}$ for some element of order $n$, it follows that for
$G\cong G^{-} \oplus C_{\exp(G)}$, one has
\begin{equation}
\label{elb_eq_2}
\begin{split}
\mathsf{D}_k(G)\ge \mathsf{D}(G^-)-1+ k \exp(G) &  \ge \mathsf{D}^{\ast}(G^-)-1+ k \exp(G) \\
&  =\mathsf{D}^{\ast}(G) +(k-1) \exp(G).
\end{split}
\end{equation}
It is known that for certain types of groups (see \cite{halterkoch92,delormeetal01}), namely for groups of rank at most $2$ and more generally for groups that have a ``large'' exponent relative to the order of the group, this bound is optimal (see \cite[Theorem 6.1.5]{geroldingerhalterkochBOOK} and cf.~Remark \ref{as_rem} for details).
However, by \cite[Lemma 3.7]{delormeetal01} (also cf.~Remark \ref{as_rem}) it is known that this bound is not always optimal.

Here we present a different construction of a lower bound, which seems to be better for groups with ``large'' rank relative to the order of the group. An explicit comparison with the bound given in \eqref{elb_eq_1}, for general $G$, is problematic, since in general one has little knowledge of $\mathsf{D}(G^-)$. However, restricting to $p$-groups, where the Davenport constant is known, or comparing with the weaker bound, involving $\mathsf{D}^{\ast}(G)$, recalled in \eqref{elb_eq_2}, this could be made precise. We omit a detailed analysis, but see the discussion given in and after Remark \ref{as_rem} for a formulation making transparent what quantities need to be compared.

We point out that our construction is in essence a construction of a lower bound for $\mathsf{D}_2(G)$, which is then extended in the obvious way. Yet, since it is convenient in Section \ref{as}, we formulate the result in this form.

\begin{theorem}
\label{elb_thm}
Let $r \in \mathbb{N}$ and $G=C_{n_1}\oplus \dots \oplus C_{n_r}$ with $1< n_1\mid \dots \mid n_r$.
Let $s\in \mathbb{N}\setminus \{1\}$  and $t \in [1,r]$ such that $s(s-1)/2 \le r - t+1$. Then $\mathsf{D}_k(G)\ge \mathsf{D}^{\ast}(G)+ s \lfloor n_t / 2 \rfloor + \delta + (k-2)n_r$ where $\delta=0$ if $n_t$ is even and $\delta=1$ if $n_t$ is odd.
\end{theorem}
\begin{proof}
By the lower bounds recalled above, we may assume that $t=1$ and $s(s-1)/2 = r$.
Let $\{e_1, \dots, e_r\}$ be a basis of $G$ where $\ord(e_i)=n_i$. Let $e_i'= (n_i/n_1)e_i$.
Let $\mathcal{P}$ denote the set of all subsets with two elements of $[1,s]$ and let $f:\mathcal{P}\to [1,r]$ be some injective map.
For $j \in [1,s]$, let $g_{j}= \sum_{P \in \mathcal{P}, \, j \in P}e_{f(P)}'$.
Let $T= (\prod_{i=1}^sg_j)^{\lfloor n_1  /2 \rfloor} g_s^{\delta}$  with $\delta$ as above and
\[S=T(\prod_{i=1}^re_i^{n_i-1}) e_r^{(k-2)n_r}.\]

We assert that $S$ is not divisible by the product of $k$ non-empty zero-sum sequences, which establishes the result.
Seeking a contradiction, we assume that there exist minimal zero-sum 
sequences $A_1, \dots, A_k$ over $G$ such that $A_1\dots A_k\mid S$.

For $i \in [1,r]$, let $\pi_i: G\to \langle e_i\rangle$ denote the projection with respect to the basis $\{e_1, \dots, e_r\}$.
We note that $\pi_i(T)\mid (e_i')^{n_1} 0^{|T|- n_1 + 1}$ for each $i \in [1,r]$; and $\pi_i(S)\mid (e_i')^{n_1} e_i^{n_i - 1} 0^{|S| - n_1 - n_i + 2}$  for each $i \in [1,r-1]$ and \[\pi_r(S)\mid (e_r')^{n_1}e_r^{(k-1)n_r-1} 0^{|S|-n_1 - (k-1)n_r + 2}.\]

We observe that for each $j \in [1,k]$ we have  $A_j=e_r^n$ or $\gcd(A_j,T)\neq 1$.
We note that $A_j \neq e_r^n$ for at least two $j \in [1,k]$, say $\gcd(A_j,T)\neq 1$ for $j \in [1,2]$.
Let $g_{k_1} \mid A_1$ and let $g_{k_2}\mid A_2$ (possibly $k_1=k_2$).
Further, let $P \in \mathcal{P}$ with $\{k_1,k_2\}\subset P$ and $i_0=f(P)$. Clearly $\pi_{i_0}(g_{k_1})= \pi_{i_0}(g_{k_2})=e_{i_0}'$.
For $i\in [1,2]$, since $\sigma(\pi_{i_0}(A_j))=0$ and $e_{i_0}'\mid \pi_{i_0}(A_j)$ it follows that
$\pi_{i_0}(A_j)$ is divisible by a non-empty zero-sum sequence over $\langle e_{i_0}\rangle \setminus \{0\}$.
Since for $i_0 \neq r$ we have $\pi_{i_0}(S) \mid (e_{i_0}')^{n_1} e_{i_0}^{n_{i_0} - 1} 0^{|S| - n_{1} - n_{i_0} + 2}$
and $(e_{i_0}')^{n_1} e_{i_0}^{n_{i_0} - 1}$ is not divisible by the product of two non-empty zero-sum sequences, it follows that
$i_0=r$.
Thus, we get that $ \pi_r(A_j)$ is divisible by zero-sum sequence over $\langle e_{r}\rangle \setminus \{0\}$ for each $j\in [1,k]$.
However,  $(e_r')^{n_1}e_r^{(k-1)n_r-1}$  is not divisible by the product of $k$ non-empty zero-sum sequences, a contradiction.
\end{proof}

\section{An asymptotic result}
\label{as}

In this section we combine the results established and recalled so far to show
that for each $G$ the sequence $(\mathsf{D}_k(G))_{k \in \mathbb{N}}$ is eventually an arithmetic progression with difference $\exp(G)$.  As mentioned in Section \ref{int}, in \cite{delormeetal01} it was already proved that the sequence $(\mathsf{D}_k(G) - k \exp(G))_{k\in \mathbb{N}}$ is bounded.
Facilitated by this result, we introduce two new invariants that seem useful when investigating $\mathsf{D}_k(G)$.

\begin{lemma}
\label{as_lem}
Let $G$ be a finite abelian group. There exist  $D_G\in \mathbb{N}_0$ and $k_G \in \mathbb{N}$ such that
$\mathsf{D}_k(G)= D_G + k \exp(G)$ for each $k \ge k_G$.
\end{lemma}
\begin{proof}
We start by establishing the following assertion:
If $k \ge \eta(G)/\exp(G) - 1=k_0$, then $\mathsf{D}_{k+1}(G)\le \mathsf{D}_{k}(G)+\exp(G)$.

By \eqref{elb_eq_2}, $\mathsf{D}_{k}(G)\ge k \exp(G) \ge \eta(G) -\exp(G)$.
Thus $\mathsf{D}_{k}(G) + \exp(G)\ge \eta(G) = \mathsf{s}_{\le \exp(G)}(G)$.
Thus by Remark \ref{ub_rem}.3, with $\ell=\exp(G)$, we have
\[\mathsf{D}_{k+1}(G) \le \max \{\mathsf{D}_k(G) + \exp(G), \mathsf{s}_{\le \exp(G)}(G)-1\}= \mathsf{D}_k(G) + \exp(G),\] as claimed.

Thus, for $k \ge k_0$, we have $\mathsf{D}_{k+1}(G)-(k+1)\exp(G)\le \mathsf{D}_k(G)-k \exp(G)$, i.e.,
the sequence $(\mathsf{D}_k(G)- k\exp(G))_{k \in \mathbb{N}}$ is eventually non-increasing.
Since by \eqref{elb_eq_2} $\mathsf{D}_k(G)- k\exp(G)\ge 0$ for each $k$, the above sequence is additionally bounded below and consequently eventually constant.
The claim follows.
\end{proof}

In view of this result the following definition makes sense.
\begin{definition}
Let $G$ be a finite abelian group.
\begin{enumerate}
  \item Let $\mathsf{D}_0(G) \in \mathbb{N}_0$ such that $\mathsf{D}_{k}(G)=\mathsf{D}_0(G)+k \exp(G)$ for all sufficiently large $k \in \mathbb{N}$.
  \item Let $k_{\mathsf{D}}(G)\in \mathbb{N}$ be minimal with $\mathsf{D}_{k}(G)=\mathsf{D}_0(G)+k \exp(G)$ for $k \ge k_{\mathsf{D}}(G)$.
\end{enumerate}
\end{definition}
The proof of Lemma \ref{as_lem} being non-constructive, we have, for general $G$, no upper bound for $k_{\mathsf{D}}(G)$. We derive one in Section \ref{rlb}.
However, known bounds on $\mathsf{D}_k(G)$ readily yield bounds for $\mathsf{D}_0(G)$ and known results on $\mathsf{D}_k(G)$ can be recast using this terminology. For illustration and to summarize most of the explicit results on $\mathsf{D}_k(G)$, we do this in the remark below. These results can be found, the first two, in \cite{halterkoch92}, \cite{delormeetal01}, \cite[Theorem 6.1.5]{geroldingerhalterkochBOOK}, the third one in \cite{delormeetal01}, and the last one in \cite{bhowmikschalge07}.

\begin{remark}
\label{as_rem}
Let $G$ be a finite abelian group with exponent $n$.
\begin{enumerate}
\item $\mathsf{D}(G^-)-1\le \mathsf{D}_0(G) \le \max \{\mathsf{D}(G)-n, \eta(G)-2n\}$.
\item If $\mathsf{r}(G)\le 2$, or more generally if $\eta(G)\le \mathsf{D}(G)+n$ and $\mathsf{D}(G)=\mathsf{D}(G^-)+n-1$, then $\mathsf{D}_0(G)=\mathsf{D}(G^-)-1$ and $k_{\mathsf{D}}(G)=1$.
\item $\mathsf{D}_0(C_2^3)= 3$ and $k_{\mathsf{D}}(C_2^3)=2$. Additionally, $\mathsf{D}_1(C_2^3)=4$.
\item $\mathsf{D}_0(C_3^3)= 6 $ and $k_{\mathsf{D}}(C_2^3)=3$. Additionally, $\mathsf{D}_1(C_3^3)=7$ and $\mathsf{D}_2(C_3^3)=11$.
\end{enumerate}
\end{remark}
Finally, we point out that Theorem \ref{elb_thm} shows that $\mathsf{D}_0(G) \ge \mathsf{D}^{\ast}(G^-) - 1 + (s \lfloor n_t / 2 \rfloor + \delta - n_r)$ with $s$ and $n_i$ as defined there. Additionally, we see by this result that $k_{\mathsf{D}}(G) > 1$ for $p$-groups with ``large'' rank, and it seems that this is the case for numerous other types of groups; again a precise statement in this regard is impeded by the lack knowledge about the Davenport constant.

\section{Recursive lower bounds for $\mathsf{D}_k(G)$ and an upper bound for $k_{\mathsf{D}}(G)$}
\label{rlb}

In this section we establish recursive lower bounds for $\mathsf{D}_k(G)$.
These are used to derive an upper bound for $k_{\mathsf{D}}(G)$. At first, we derive a bound that involves the successive distance $\delta(G)$ and thus is not (yet) explicit. However, in Theorem \ref{rlb_thm_sd} we establish an explicit upper bound for $\delta(G)$ to resolve this issue.

On the one hand, as mentioned in Section \ref{prel} we have $\mathsf{D}_{k+1}(G)\ge \mathsf{D}_{k}(G) +1$.
On the other hand, in view of Remark \ref{ub_rem}.3, it is clear that without restriction on $k$, the best bound of the form $\mathsf{D}_{k+1}(G) \ge \mathsf{D}_{k}(G) + c$ that possibly can hold, for each $k$, would be $\mathsf{D}_{k+1}(G) \ge \mathsf{D}_{k}(G) + \exp(G)$.

First, we slightly improve the former bound, except for $|G|=1$ where it is obviously optimal.
In view of the discussion above, the following result is optimal, for large $k$, for elementary $2$-groups.

\begin{lemma}
\label{rlb_lem_lb2}
Let $G\neq \{0\}$ be a finite abelian group.
Let $B$ be a zero-sum sequence over $G$ and let $g \in G\setminus \{0\}$.
Then $\max \mathsf{L}(B(-g)g)= 1+\max \mathsf{L}(B)$.
In particular, $\mathsf{D}_{k+1}(G)\ge \mathsf{D}_{k}(G)+ 2$ for each $k \in \mathbb{N}$. Moreover, if $C$ is a zero-sum sequence 
over $G$ with $\max \mathsf{L}(C)\le k$ and $|C|= \mathsf{D}_k(G)$, then $0 \notin \supp(C)$.
\end{lemma}
\begin{proof}
Let $\zeta$ be a factorization of $B(-g)g$ with maximal length, that is $|\zeta|=\max  \mathsf{L}(B(-g)g)$.
If the minimal zero-sum sequence $(-g)g$ appears in 
this fatorization, i.e., $\mathsf{v}_{(-g)g}(\zeta)> 0$, then $((-g)g)^{-1}\zeta$ is a factorization of $B$ and the claim follows.
Thus suppose there exist minimal zero-sum sequences $U_1$ and $U_2$ such that $\mathsf{v}_{g}(U_1)>0$, $\mathsf{v}_{-g}(U_2)>0$, and $U_1U_2 \mid \zeta$.
Since $(-g)g \mid U_1U_2$ (as sequences, not factorizations), there exists some zero-sum sequence $C$ such that
$U_1U_2=((-g)g)C$. We assert that $C$ is a minimal zero-sum sequence.
Let $\zeta_C$ be a factorization of $C$. It follows that $(U_1U_2)^{-1}((-g)g)\zeta_C \zeta$ is a factorization of $B(-g)g$.
Thus, $|\zeta|-2+|\zeta_C|+1\in \mathsf{L}(B(-g)g)$. Since $|\zeta|=\max  \mathsf{L}(B(-g)g)$, it follows that $|\zeta_C|=1$.
The claim $\mathsf{D}_{k+1}(G)\ge \mathsf{D}_{k}(G)+ 2$ follows immediately.

To prove the final claim, we consider some 
zero-sum sequence $C$ with $\max \mathsf{L}(C)\le k$ and  $|C|= \mathsf{D}_k(G)$, and assume to the contrary that $0\mid C$.
It follows that $\max\mathsf{L}(0^{-1}C)\le k-1$ and thus, for $g \in G \setminus \{0\}$, we have $\max \mathsf{L}(0^{-1}C(-g)g)$.
Yet, $|0^{-1}C(-g)g|> |C|$, a contradiction.
\end{proof}

Now, we show that indeed for each finite abelian group $G$, $\mathsf{D}_{k+1}(G) = \mathsf{D}_{k}(G) + \exp(G)$ for all sufficiently large $k$.
We point out that it is well possible that the inequality $\mathsf{D}_{k+1}(G) \ge \mathsf{D}_{k}(G) + \exp(G)$ in fact holds for any $k$, yet our proof requires that $k$ is large.
Then, we make the condition ``sufficiently large'' explicit and thus establish an explicit upper bound for $k_{\mathsf{D}}(G)$; however, little effort is made to optimize this bound, since it seems very unlikely that we obtain a value close to the actual value using the present approach.

\begin{proposition}
\label{rlb_prop_exp}
Let $G$ be a finite abelian group and let \[k \ge  \delta(G)\exp(G)|\mathcal{A}(G)|+ (\eta(G)- \mathsf{D}(G^-)).\]
Then $\mathsf{D}_{k+1}(G)=\mathsf{D}_{k}(G)+\exp(G)$.
\end{proposition}
For clarity of exposition and since the technical results might be of independent interest, we first establish some auxiliary results.

\begin{lemma}
\label{rlb_lem_exp1}
Let $G$ be a finite abelian group and let $k \in \mathbb{N}$.
Let $B$ be a zero-sum sequence over $G$ with $\max\mathsf{L}(B)=k$.
If there exist some factorization $\zeta$ of $B$ with $|\zeta|=k$ 
and some minimal zero-sum sequence $U$ such that
$\mathsf{v}_{U}(\zeta)\ge \delta(G)$, i.e., $U$ occurs at least $\delta(G)$
 times in this factorization, then $\max \mathsf{L}(BU)=k+1$.
\end{lemma}
\begin{proof}
Clearly $\zeta U$ is a factorization of $BU$ and thus $k+1\in \mathsf{L}(BU)$.
It remains to show that $\max \mathsf{L}(BU)=k+1$.
Assume not. Let $k'= \min \{n \in \mathsf{L}(BU)\colon n > k+1\}$.
Since $k+1$ and $k'$ are successive distances of $\mathsf{L}(BU)$ and by the definition of $\delta(G)$
there exists some factorization $\zeta'$ of $BU$ with $|\zeta'|=k'$ such that $\mathsf{d}(\zeta U, \zeta')\le \delta(G)$.
Since $\mathsf{v}_{U}(\zeta U)\ge 1 + \delta(G)$, it follows that $\mathsf{v}_U(\zeta')\ge 1$.
Consequently $U^{-1}\zeta'$ is a factorization of $B$ and thus $k' - 1 = |U^{-1}\zeta'| \in \mathsf{L}(B)$.
Since $k'-1> k$, this contradicts $\max \mathsf{L}(B)= k$.
Thus, we have $\max \mathsf{L}(BU)=k+1$ and the claim is proved.
\end{proof}

\begin{lemma}
\label{rlb_lem_exp2}
Let $G$ be a finite abelian group and let $k \in \mathbb{N}$.
Let $B$ be a zero-sum sequence over $G$ such that $\max\mathsf{L}(B)=k$ and $|B|=\mathsf{D}_k(G)$.
There exists some $g \in G$ with $\ord(g)=\exp(G)$ and some
factorization $\zeta_g$ of $B$ with $|\zeta_g|=k$ such that
$\mathsf{v}_{g^{\ord(g)}}(\zeta_g)\ge \lfloor(k - (\eta(G) - \mathsf{D}(G^-)))/(\exp(G)|\mathcal{A}(G)|)\rfloor $.
\end{lemma}
\begin{proof}
Let $\zeta$ be a factorization of $B$ with $|\zeta|=k$. Let $\zeta=\zeta^{<}\zeta^{=}\zeta^{>}$ where $\zeta^{<}$, $\zeta^{=}$, $\zeta^{>}$ consists of those atoms with lengths less than, equal to, greater than the exponent, resp. Moreover, we set $\zeta^{\ge} = \zeta^{=}\zeta^{>}$.
We start by establishing the following assertion.

\noindent
\textbf{Assertion:} $|\zeta^{<}|\le \eta(G)- \mathsf{D}(G^-)$.\\
\emph{Proof of Assertion:}
We recall that $\pi(\xi)$, for a factorization $\xi$,
is a zero-sum sequence; thus $|\pi(\xi)|$ denotes the length of this sequence (the number of elements from $G$ in this sequence),
while $|\xi|$ denotes the length of the factorization (the number of 
minimal zero-sum sequence in this factorization). 
We observe that $|\zeta^{\ge}|= \max \mathsf{L}(\pi(\zeta^{\ge}))$.
By Remark \ref{ub_rem}.2 we thus get
$|\zeta^{\ge}|\ge (|\pi(\zeta^{\ge})|-\eta(G)+1)/\exp(G)$, thus $|\pi(\zeta^{\ge})|\le  \exp(G) |\zeta^{\ge}| +\eta(G)-1$ .
Since $|\zeta^{\ge}|+|\zeta^{<}|= k$,
we get $|\pi(\zeta^{\ge})|\le \exp(G) (k-|\zeta^{<}|) + \eta(G)-1$ and
$\exp(G) |\zeta^{<}|+  |\pi(\zeta^{\ge})|\le \exp(G) k + \eta(G)-1$.
We observe that $|\pi(\zeta^{<})|\le |\zeta^{<}| (\exp(G)-1)$.
Thus, we get
$\exp(G) k + \eta(G)-1\ge |\zeta^{<}|+ |\pi(\zeta^{<})|+  |\pi(\zeta^{\ge})|= |\zeta^{<}|+|B|$.
By  \eqref{elb_eq_2} we have $|B| \ge \mathsf{D}(G^-)-1+ k \exp(G)$
and consequently
$\exp(G) k + \eta(G)-1\ge   |\zeta^{<}|+|B| \ge |\zeta^{<}|+ \mathsf{D}(G^-) - 1 +  k \exp(G)$ and
$|\zeta^{<}|\le \eta(G) - \mathsf{D}(G^-)$, proving the assertion.

We note that there exists some minimal zero-sum sequence $U\in \mathcal{A}(G)$ such that $\mathsf{v}_U(\zeta^{\ge})\ge  |\zeta^{\ge}|/ |\mathcal{A}(G)|$.
Let $\ell \in \mathbb{N}_0$ maximal such that $\mathsf{v}_U(\zeta^{\ge})\ge \ell \exp(G)$. We note that
\[
\ell \ge \lfloor|\zeta^{\ge}|/ (|\mathcal{A}(G)|\exp(G))\rfloor \ge \lfloor(k - (\eta(G)-  \mathsf{D}(G^-)))/(\exp(G)|\mathcal{A}(G)|)\rfloor.
\]
Thus, it suffices to show that for some $g \in G$ with $\ord(g)=\exp(G)$, we have $\mathsf{v}_{g^{\ord(g)}}(\zeta_g)\ge \ell$.
For $\ell = 0$ there is nothing to prove. Thus, suppose $\ell > 0$.
Let $\zeta_e = U^{\exp(G)}\in \mathsf{Z}(G)$, i.e., the factorization
consisting of $U$ repeated $\exp(G)$ times.
Let $\zeta_g' = \prod_{h \in G} (h^{\ord(h)})^{\exp(G)\mathsf{v}_h(U)/\ord(h)}\in \mathsf{Z}(G)$.
Then $\pi(\zeta_g') = \pi(\zeta_e)$ and $|\zeta_g'| = \exp(G)\mathsf{k}(U) \ge |U| \ge \exp(G) = |\zeta_e|$.
Thus $\zeta_g = \zeta_e^{-\ell}\zeta_g'^{\ell}\zeta$ is a factorization of $B$ and $|\zeta_g|\ge |\zeta|$.
By the maximality of $|\zeta|$ it follows that indeed $|\zeta_g|= |\zeta|$, i.e., $|\zeta_g'| = \exp(G)$ implying that $\ord(h) = \exp(G)$ for each $h \in \supp(U)$ and $|U| = \exp(G)$.
Thus, for $g \in \supp(U)$ we have $\mathsf{v}_{g^{\ord(g)}}(\zeta_g)\ge \ell$ implying the claim.
\end{proof}

\begin{proof}[Proof of Proposition \ref{rlb_prop_exp}]
The main point is to show that $\mathsf{D}_{k+1}(G) \ge \mathsf{D}_k(G) + \exp(G)$. 
Let $B$ be a zero-sum sequence over $G$ such that $\max \mathsf{L}(B)=k$ and $|B|=\mathsf{D}_k(G)$.  By Lemma \ref{rlb_lem_exp2}
there exists some $g \in G$ with $\ord(g)=\exp(G)$ and some factorization $\zeta_g$ of $B$ with $|\zeta_g|=k$ such that
\[
\mathsf{v}_{g^{\ord(g)}}(\zeta_g)\ge \lfloor(k - (\eta(G) - \mathsf{D}(G^-)))/(\exp(G)|\mathcal{A}(G)|)\rfloor \ge \delta(G).
\]
By Lemma \ref{rlb_lem_exp1} applied with $U=g^{\ord(g)}$,
it follows that $\max \mathsf{L} (Bg^{\ord(g)}) = k + 1$.
Thus, $\mathsf{D}_{k+1}(G)\ge |Bg^{\ord(g)}| = \mathsf{D}_{k}(G)+\exp(G)$.

To complete the proof, it remains to assert that 
$\mathsf{D}_{k+1}(G) \le \mathsf{D}_k(G) + \exp(G)$.
This follows immediately by Proposition \ref{ub_prop_bas}.3, provided that $k \ge \eta(G) - \exp(G) -1$.
By our assumption on $k$, this is clear for $|G|\le 2$, 
and for $|G|\ge 3$ we note that $\delta(G)\neq 0$ and $\mathsf{D}(G^{-})\le \mathsf{D}(G)\le |\mathcal{A}(G)|$, implying the condition. 
\end{proof}

In combination with Remark \ref{ub_rem}.3, Proposition \ref{rlb_prop_exp} directly yields an upper bound for $k_{\mathsf{D}}(G)$ (cf.~Proof of Theorem \ref{rlb_thm_expl}).
Yet, without further investigation, this bound is not explicit.
So far it is only known that $\delta(G)$ is finite (cf.~Section \ref{prel}), yet no explicit bound is known (the other quantities that are involved in the above upper bound can be easily replaced by explicit upper bounds, cf.~Proof of Theorem \ref{rlb_thm_expl}).
Thus, we establish an explicit upper bound for $\delta(G)$; this is of some
independent interest as it yields information on the analog of 
that invariant for Krull monoids with finite class group (see \cite{GGSS} for recent results on this invariant for Krull monoids with infinite cyclic class group). 
 To do so, we combine the proof of its finiteness by A.~Foroutan \cite{foroutan06} with a result of P.~Diaconis, R.~L.~Graham, and B.~Sturmfels \cite{diaconisetal96};  our goal is a simple bound, thus even with the present proof some improvement could be achieved easily, e.g., by not simplifying certain expressions
or by using better bounds for the Davenport constant.

\begin{theorem}
\label{rlb_thm_sd}
Let $G$ be a finite abelian group. Then $\delta(G)\le  (2|G|)^{3|G|+1}$.
\end{theorem}
\begin{proof}
The result is trivial for $|G| \le 2$, since in this case $\Delta(G)=\emptyset$. Thus, we assume that $|G| \ge 3$.
We follow \cite{foroutan06}, yet for technical reasons we treat $d$ and $-d$ separately. For $d\in \mathbb{Z}$ such that $|d| \in \Delta(G)$, let $\delta(d)$ denote the maximum over all
$\max\{|\zeta|,|\zeta'|\}$ where $(\zeta,\zeta')\in \mathsf{Z}(G)\times \mathsf{Z}(G)$ with $\pi(\zeta)=\pi(\zeta')$ and $|\zeta|=|\zeta'|+d$ (i.e., we consider pairs of factorizations of the same zero-sum sequences such 
that the length of the factorizations differ by a prescribed value), and $(\zeta,\zeta')$ is minimal with this property (i.e., for each proper $\xi\mid \zeta$ and $\xi'\mid \zeta'$ we have $\pi(\xi)\neq \pi(\xi')$ or $|\xi| \neq |\xi'|+d$).

By \cite{foroutan06}, it is known that $\delta(G)\le \max\{\delta(d)\colon d \in \mathbb{Z}, \, |d| \in \Delta(G)\}$.

We assume to the contrary that $\delta(d)>  (2|G|)^{3|G|+1}$ for some $d\in \mathbb{Z}$ with $|d| \in \Delta(G)$.
Let $(\zeta,\zeta')\in \mathsf{Z}(G)\times \mathsf{Z}(G)$, fulfilling the above conditions, that attains the maximum $\delta(d)$.  By the minimality, we get that $0 \nmid \zeta$ and $0\nmid \zeta'$.
Let $\varphi: \mathcal{F}(G\setminus \{0\}) \to \mathbb{Z}^{G\setminus \{0\}}$, denote the map defined via $\varphi(S)=(\mathsf{v}_g(S))_{g \in G\setminus \{0\}}$, and let
$\varphi^{\ast}: \mathcal{F}(G\setminus \{0\})\to \mathbb{Z}^{G\setminus \{0\}}\times \mathbb{Z}$ be defined via $\varphi^{\ast}(S)=(\varphi(S),1)$.

Let $\zeta= \prod_{i=1}^{n}A_i$ and $\zeta'= \prod_{i=1}^{n+d}A_i'$ with $A_i,A_i'\in \mathcal{A}(G)$.
Then $\sum_{i=1}^{n}\varphi^{\ast}(A_i) + (\overline{0},d )= \sum_{i=1}^{n+d}\varphi^{\ast}(A_i')$.

By \cite[Theorem 1]{diaconisetal96}, applied, say, to the set $\mathcal{A}=\{\varphi^{\ast}(S)\colon S\in \mathcal{F}(G\setminus \{0\}),\, |S|\le \mathsf{D}(G)\}\cup \{(\overline{0},d)\}\subset  \mathbb{Z}^{G\setminus \{0\}}\times \mathbb{Z}$,
we get that this relation is not minimal; the bound that is provided by that result, for $|\zeta|+|\zeta'|+1$, is $(2|G|)^{|G|}(|G|+1)^{|G|+1}\mathcal{D}$, where $\mathcal{D}$ is the maximum of the absolute values of the determinants of the $|G|\times |G|$ minors of the $|\mathcal{A}|\times |G|$ matrix $(a \colon a \in \mathcal{A})$. Using well-known bounds on the determinant, e.g., Hadamard's inequality (replacing the Euclidean norm by the $1$-norm) or more directly the bound of C.R.~Johnson and M.~Newman \cite{johnsonnewman80} (also see \cite{lev08} for an alternate proof), we get
$\mathcal{D}\le |\mathsf{D}(G)+1|^{|G|}$ (note that $|d|\le \max \Delta(G)\le \mathsf{D}(G)-2$). Using $\mathsf{D}(G)\le |G|$ and performing some immediate simplifications, the claim is established.

Thus, there exists some $\emptyset \neq I\subsetneq [1,n]$   and $\emptyset \neq J\subsetneq [1,n + d]$ such that
 $\sum_{i\in I}\varphi(A_i) + (\overline{0},d )= \sum_{j \in J}\varphi(A_j')$; note that we may assume that $(\overline{0},d )$ is part of the sum, since otherwise we could consider the complements, and that still $I$ has to be non-empty, since $(\overline{0},d )$ can not be the only element involved in such an equality.
Setting $\xi= \prod_{i\in I}A_i$ and $\xi'= \prod_{j\in J}A_j$ we get that $|\xi|= |I|=|J|+d= |\xi'|+d$ and
$\pi(\xi)=\pi(\xi')$ a contradiction to the minimality of $(\zeta,\zeta')$.
\end{proof}

Now, we combine the results to get an upper bound for $k_{\mathsf{D}}(G)$.

\begin{theorem}
\label{rlb_thm_expl}
Let $G$ be a finite abelian group. Then \[k_{\mathsf{D}}(G)\le \delta(G)\exp(G)|\mathcal{A}(G)|+ \eta(G)-\mathsf{D}(G^-).\]
In particular, $k_{\mathsf{D}}(G)\le (2|G|)^{4|G|+2}$.
\end{theorem}
\begin{proof}
By Proposition \ref{rlb_prop_exp} we know that for $k\ge  \delta(G)\exp(G)|\mathcal{A}(G)|+ (\eta(G)- \mathsf{D}(G^-))$ we have $\mathsf{D}_{k+1}(G) = \mathsf{D}_k(G) + \exp(G)$.

To get the ``in particular'' statement, we use the upper bound on $\delta(G)$ derived in Theorem \ref{rlb_thm_sd}, the crude bound  $|\mathcal{A}(G)|\le |G|^{\mathsf{D}(G)}\le |G|^{|G|}$, $\exp(G) \le |G|$, and recall that $\eta(G)\le |G|$.
\end{proof}
In Section \ref{e2g} we establish a much better bound for $k_{\mathsf{D}}(G)$ for elementary $2$-groups.

\section{$\mathsf{D}_k(G)$ for elementary $2$-groups}
\label{e2g}

In this section we consider the $k$-wise Davenport constants for elementary $2$-groups, to illustrate and complement the general results of the preceding sections with more explicit results. As indicated in Section \ref{int}, elementary $2$-groups are an interesting class of groups in this regard, since on the one hand the gap between $\mathsf{D}(G)$ and $\eta(G)$ is in general quite large, the former being $\mathsf{r}(G)+1$ and the latter $|G| = 2^{\mathsf{r}(G)}$ (cf.~Remark \ref{ub_rem} for the relevance of these invariants in this context), and on the other hand investigations of zero-sum problems in this type of group are simplified due to fact that the structure of minimal zero-sum sequences over these groups is simple and known precisely (cf.~below).

Still, it seems that to determine $k$-wise Davenport constants of elementary $2$-groups is a challenging task, and we are only able to obtain improved (relative to the general case) bounds for these constants and to determine them in special cases.
The results on the values of $\mathsf{D}_k(G)$ are given in Subsection \ref{e2g_subs_main}. In Subsection \ref{e2g_subs_aux} we recall and obtain various results that are needed in these investigations.

\subsection{Technical results}
\label{e2g_subs_aux}

We collect some well-known facts on the structure of (minimal) zero-sum sequences over elementary $2$-groups that will be used frequently and without reference.  For a detailed investigation of this type of questions, for general finite abelian groups, see, e.g., \cite{gaogeroldinger99} and \cite{gaogeroldingersurvey}.
Let $r \in \mathbb{N}$.
Since $C_2^r$ is (in a natural way) a vector space over the field with two elements, a sequence $S\in \mathcal{F}(C_2^r)$ has no non-empty zero-sum subsequence if and only if $\supp(S)$ is (linearly) independent. Thus, each minimal zero-sum sequence over $C_2^r$, except $0$, is of the form $(e_1+\dots+e_s)\prod_{i=1}^se_i$ with independent $e_i$s. And, $\mathsf{D}_1(C_2^r)=r+1$.

A sequence $S\in \mathcal{F}(C_2^r)$ has no zero-sum subsequence of length $1$ if and only if $0 \nmid S$ (clearly this is true for any finite abelian group); moreover, $S$ has no zero-sum subsequence of length $2$ if and only if $S$ is squarefree. In particular, $\eta(C_2^r)=2^r$.

By the remark on the structure of minimal zero-sum sequences, it follows that there exists a squarefree minimal zero-sum sequence over $C_2^r$ of length $s$ if and only if $s\in [3, r+1]$. Since, for $r\ge 2$, $\sum_{g \in C_2^r}g=0$, the existence and non-existence of squarefree zero-sum sequences of length close to $|C_2^r|$ often can be decided easily using this fact and the preceding remark (via considering the squarefree sequence of elements not contained in the original sequence).

We continue with the following simple observation.
\begin{lemma}
\label{e2g_lem_l3}
Let $r\in \mathbb{N}$ and let $S$ be a sequence over $ C_2^r$.
Then, $S$ has no non-empty zero-sum subsequence of length at most $3$ if and only if
$S$ is squarefree and $\supp(S)$ is sum-free.
Additionally, $S$ has no non-empty zero-sum subsequence of lengths at most $4$ if and only if
$S$ has no non-empty zero-sum subsequence of lengths at most $3$ and $\supp(S)$ is a Sidon set.
\end{lemma}
\begin{proof}
We observe that if $\supp(S)$ is sum-free, then $0 \notin \supp(S)$. In view of the remarks above it remains to consider squarefree $S$ with $0\nmid S$.
We note that for $g_1,g_2,g_3 \in C_2^r\setminus \{0\}$ we have $\sigma(g_1g_2g_3)=0$ if and only if $g_1+g_2=g_3$; and, since none of the elements is $0$, $g_1+g_2=g_3$ implies that the elements are distinct. And, the first claim follows.
To see the second claim, it suffices to note that for $g_1,g_2,g_3,g_4\in C_2^r$ we have $\sigma(g_1g_2g_3g_4)=0$ if and only if $g_1+g_2=g_3+g_4$.
\end{proof}

This observation is useful for our investigations, since the structure of sufficiently large sum-free sets in $C_2^r$ is known precisely.
We recall a special case of a result due to A.~A. Davydov and L.~M. Tombak \cite{davydovtombak89}, as given in \cite{grynkiewiczlev}; for ease of application in the following arguments, we rephrase it, via the characterization given in Lemma \ref{e2g_lem_l3}, in the way it is applied in the present paper.

\begin{theorem}
\label{e2g_thm_l3ex}
Let $r \in \mathbb{N}$. Let $S$ be a squarefree sequence over $C_2^r$ with $0 \nmid S$ and $|S|\ge 9 ( 2^{r-5} )$.
Then, the following statements are equivalent
\begin{itemize}
  \item $S$ has no non-empty zero-sum subsequence of length $3$.
  \item $\supp(S)$ is contained in the non-zero coset of a subgroup of index $2$ or $\supp(S)$ is contained in $\{e_1,e_2,e_3,e_4,(e_1+e_2+e_3+e_4)\}+G'$ where
$G'$ is a subgroup of index $16$ and $C_2^r = \langle e_1, \dots, e_4\rangle \oplus G'$.
\end{itemize}
In particular, $\mathsf{s}_{\le 3}(C_2^r)=1 + 2^{r-1}$.
\end{theorem}
We point out that, except once in the proof of Proposition \ref{e2g_prop_small}, we only use this result for sequences of length greater than $5 ( 2^{r-4} )$, thus avoiding the second type of set in the characterization.

The above arguments can also be used in the converse direction.
\begin{lemma}
\label{e2g_lem_lb}
Let $r\in \mathbb{N}$.
\begin{enumerate}
  \item Let $B$ be squarefree zero-sum sequence over $C_2^r$ with $0\nmid B$. Then $\max \mathsf{L}(B)\le |B|/3$.
  \item Let $G' \subsetneq C_2^r$ be a subgroup and let $e \in C_2^r\setminus G'$. Let $B$ be a squarefree zero-sum sequence over $C_2^r$ with $\supp(B)\subset e+G'$.
  Then, $\max \mathsf{L}(B)\le |B|/4$.
\end{enumerate}
\end{lemma}
\begin{proof}
Let $B=A_1 \dots A_k$ with minimal zero-sum sequences $A_i$. Since for each $i$ we have $|A_i| \ge 3$ and $|A_i| \ge 4$, resp., the claim follows.
\end{proof}

Next we establish a simple upper bound for $\mathsf{s}_{\le k}(C_2^r)$, for even $k$.
By Lemma \ref{e2g_lem_l3}, the bound for $k=4$ follows by results on Sidon sets and the general case is an immediate modification of that argument; for clarity we include the short argument.
We point out that by a result of B.~Lindstr{\"o}m \cite{lindstrom}, for $k=4$, this bound is close to optimal.

\begin{lemma}
\label{e2g_lem_D2m}
Let $r \in \mathbb{N}$ and $m \in \mathbb{N}\setminus \{1\}$. Then $\mathsf{s}_{\le 2m}(C_2^r)\le (m-1)+ (m!\ 2^r)^{1/m}$.
\end{lemma}
\begin{proof}
Let $S \in \mathcal{F}(C_2^r)$ with $|S| \ge (m-1)+ (m!\ 2^r)^{1/m}$.
We have to show that $S$ has a non-empty zero-sum subsequence of length at most $2m$.
Since otherwise the claim follows trivially, we may assume that $S$ is squarefree and does not contain $0$.
By our assumption on $|S|$ it follows that $\binom{|S|}{m}\ge 2^r$. Thus, there exist two distinct subsequence $T_1,T_2$ of $S$ of length $m$ such that $\sigma(T_1)=\sigma(T_2)$ or there exists some subsequence $T$ of $S$ of length $m$ such that $\sigma(T)=0$.
If the latter holds we are done. Thus, assume the former holds. We note that $(\gcd(T_1,T_2))^{-2}T_1T_2$, i.e. we discard all terms common to $T_1$ and $T_2$,  is a zero-sum subsequence of $S$, which is non-empty and has length at most $2m$.
\end{proof}

The following result shows that for various questions it is possible to restrict to squarefree sequences.

\begin{proposition}
\label{e2g_prop_tech}
Let $r \in \mathbb{N}\setminus \{ 1 \}$.
\begin{enumerate}
  \item $k_{\mathsf{D}}(C_2^r) = \min \{k \in \mathbb{N} \colon \mathsf{D}_0(C_2^r) = \mathsf{D}_k(C_2^r) - 2k \}$.
  \item If $B$ is a sequence over $C_2^r$ with 
$\max\mathsf{L}(B)\le k_{\mathsf{D}}(C_2^r)$ and $|B| = \mathsf{D}_{k_{\mathsf{D}}(C_2^r)}(C_2^r)$, then $B$ is squarefree and $0\nmid B$.
\end{enumerate}
In particular, $k_{\mathsf{D}}(C_2^r) \le \lfloor (2^r-1)/3 \rfloor $.
\end{proposition}
\begin{proof}
By Lemma \ref{rlb_lem_lb2} we know that $\mathsf{D}_{k+1}(C_2^r) \ge \mathsf{D}_{k}(C_2^r) +2 $ for each $k \in \mathbb{N}$.
Thus, $\mathsf{D}_0(C_2^r)\ge \mathsf{D}_{k}(C_2^r) - 2k$ for each $k\in \mathbb{N}$ and the first claim follows.

Now let $k_1 = k_{\mathsf{D}}(C_2^r)$  and let $B$
be a zero-sum sequence with $\max \mathsf{L}(B)=k_1$ and $|B|=\mathsf{D}_{k_1}(C_2^r) = \mathsf{D}_0(C_2^r)+2{k_1}$.
We note that $\mathsf{D}_{k}(C_2^r) < \mathsf{D}_0(C_2^r)+2{k}$ for $k \in [1,k_1 - 1]$.

By Lemma \ref{rlb_lem_lb2} we know $0 \nmid B$. Let $B=B'T^2$ with $B'\in \mathcal{B}(C_2^r)$ squarefree and $T \in \mathcal{F}(C_2^r)$.
We note that by Lemma \ref{rlb_lem_lb2}, we have $\max \mathsf{L}(B)= \max\mathsf{L}(B')+|T|$. Thus, we have $|B|- 2\max \mathsf{L}(B)=|B'|  - 2\max \mathsf{L}(B')$. Let $k_2= \max \mathsf{L}(B')$. By definition we have $\mathsf{D}_{k_2}(C_2^r)\ge |B'|$ and thus $\mathsf{D}_{k_2}(C_2^r)\ge \mathsf{D}_0(C_2^r) + 2 k_2$. It follows that $k_2=k_1$, i.e., $|T|=0$.
Thus $B$ is squarefree and the second claim is established. By Lemma \ref{e2g_lem_lb} $k_1 \le |B|/3\le (2^r-1 )/3$, implying the additional claim.
\end{proof}

Now we determine the maximal length of the squarefree sequence consisting of all non-zero elements of $C_2^r$.
\begin{proposition}
\label{e2g_prop_maxfull}
Let $r\in\mathbb{N}\setminus \{1\}$. Let $B$ be the squarefree zero-sum sequence with support $C_2^{r}\setminus \{0\}$.
Then $\max \mathsf{L}(B) = \lfloor (2^r - 1)/3 \rfloor $.
\end{proposition}

To prove this result we use a special case of a result of L.J.~Paige \cite{paige47} (the actual result asserts the existence of a bijection for any finite abelian group that does not have exactly one element of order $2$). Still, we give a full proof of this special case, since its details are relevant for later arguments, and we thus phrase it in a way that highlights these details. Results on refined questions of this type, e.g., on Snevily's conjecture and related questions, recently received considerable attention (for instance, see \cite{arsovski,dasguptaetal01,fengsunxian,harcos} and the references given there).

\begin{lemma}
\label{e2g_lem_full}
Let $r \in \mathbb{N}\setminus \{1\}$. There exists a bijection $\varphi: C_2^r \to C_2^r$ such that
$\{ g + \varphi(g) \colon g \in C_2^r \} = C_2^r$.
\end{lemma}
\begin{proof}
We prove the result by induction on $r$. Yet, since assuming the result holds for $r$, we can only establish that it holds for $r+2$, we first need to establish it for $2$ and $3$.
Let $r=2$. A bijection with unique fixed point $0$ has this property.
Let $r=3$. Let $f_1,f_2,f_3$ be independent. The bijection defined via $0\mapsto 0$, $f_1\mapsto f_1+f_2$, $f_2\mapsto f_2+f_3$, $f_3\mapsto f_1$,
$f_1+f_2\mapsto f_1+f_3$, $f_1+f_3\mapsto f_2$, $f_2+f_3\mapsto f_1+f_2+f_3$, and $f_1+f_2+f_3\mapsto f_3$ has this property.

Assume the claim holds for some $r \ge 2$. We consider $C_2^{r+2}$. Let $G'$ be a subgroup of index $4$ and let $e_1,e_2\in C_2^{r+2}$ such that $C_2^{r+2}= G' \oplus \langle e_1,e_2\rangle $. Let $\varphi_{G'}:G' \to G'$ a bijection such that $\{ g + \varphi_{G'}(g) \colon g \in G' \} = G'$, which exists by assumption, and let $\varphi_2: \langle e_1,e_2 \rangle \to \langle e_1,e_2 \rangle$ a bijection such that $\{ g + \varphi_2(g) \colon g \in \langle e_1, e_2\rangle \} = \langle e_1, e_2\rangle$, which exists by the case ``$r=2$''.
Let $\varphi = \varphi_2\oplus \varphi_{G'}$, i.e., $\varphi(g)=\varphi_2(e)+\varphi_{G'}(h)$ where $g=e+h$ with $e\in \langle e_1,e_2 \rangle$ and $h \in G'$.
Then, $\varphi$ is bijective and has the required property.
\end{proof}

\begin{proof}[Proof of Proposition \ref{e2g_prop_maxfull}]
By Lemma \ref{e2g_lem_lb} $\max \mathsf{L}(B) \le (2^r - 1)/3$.
Thus, it suffices to show that $B$ has some factorization of length at least $\lfloor (2^r - 1)/3 \rfloor $.
We proceed by induction.
For $r=2$ the claim is trivial, and for $r=3$ it is obvious, since $B$ is not a minimal zero-sum sequence.

Assume the claim holds for some $r \ge 2$. We consider the problem for $C_2^{r+2}$.
Let $G'$ be a subgroup of index $4$ and let $e_1,e_2\in C_2^{r+2}$ such that $C_2^{r+2}=G'\oplus \langle e_1,e_2\rangle$.
Moreover, let $\varphi: G' \to G'$ be a bijection such that $\{ g + \varphi_{G'}(g) \colon g \in G' \} = G'$, which exists by Lemma \ref{e2g_lem_full}.
Let $C=\prod_{g \in C_2^{r+2}\setminus G'}g$ and $D=\prod_{g \in G'\setminus \{0\}}g$. We note that $C$ and $D$ are zero-sum sequences and  $B=CD$.
For $h \in G'$ let $A_h= (e_1+h)(e_2+ \varphi(h))(e_1+e_2+ (h + \varphi(h)))$; we note that $A_h$ is a minimal zero-sum sequence.
Moreover, $\prod_{h \in G'}A_h= C$. Thus $2^r=|G'|\in \mathsf{L}(C)$.
And, by assumption $\max \mathsf{L}(D)= \lfloor (2^{r}-1)/3 \rfloor$. Thus, $2^r+ \lfloor (2^{r}-1)/3 \rfloor \in \mathsf{L}(B)$ and the claim is established.
\end{proof}

\subsection{Main results}
\label{e2g_subs_main}

In this section we state and prove our results on $\mathsf{D}_k(C_2^r)$.
We start by establishing an upper bound for $\mathsf{D}_2(C_2^r)$.

\begin{theorem}
\label{e2g_thm_ub2}
Let $r \in \mathbb{N}$. Then $\mathsf{D}_2(C_2^r) < (3r + 6)/2$.
\end{theorem}
\begin{proof}
We proceed by contradiction. Let $B$ be a zero-sum sequence over $C_2^r$ with
$\max\mathsf{L}(B)=2$ and $|B|=\mathsf{D}_2(C_2^r)$ and suppose $|B|\ge (3r+6)/2$. By \eqref{elb_eq_1} we have $\langle \supp(B)\rangle = C_2^r$, and by Proposition \ref{ub_prop_bas} we get  $0\nmid B$, and $B$ is squarefree. In particular, we get that $r\ge 3$. Let $A \in \mathcal{A}(C_2^r)$ with $A\mid B$, and we assume that $|A|$ is minimal among the lengths of all minimal zero-sum sequences that divide $B$.
Since $A' =A^{-1}B$ is a minimal zero-sum sequence as well, we get that $|A| \ge |B|- \mathsf{D}(C_2^r)=|B| - r-1$.

We consider three cases.

\noindent
\textbf{Case 1:}
Suppose that $|A|\ge |B| - r +1$. Since $|A'|< r + 1$  and thus $\mathsf{r}(\langle \supp(A')\rangle)< r$, there exist independent elements $e_1, \dots, e_{|A'|}$ such that
$A'=f \prod_{i=1}^{|A'|-1} e_i$ and $A= e_{|A'|}\prod_{j=1}^{|A|-1}f_j$.
Let $L \mid \prod_{j=1}^{|A|-1}f_j$ be a non-empty subsequence of minimal length such that
$\sigma(L) \in \langle e_1, \dots, e_{|A'|} \rangle$. Note that since $\sigma ( \prod_{j=1}^{|A|-1}f_j)=e_{|A|'}$ such a sequence exists.
Moreover, we note that $|L| \le \mathsf{D}(C_2^r / \langle e_1, \dots, e_{|A'|} \rangle ) = 1 + (r - |A'|)$.
Let $J \subset [1, |A'|]$ such that $\sigma(L)= \sum_{j \in J}e_j$.
Then
\[
T = L  \prod_{j \in J } e_{j}
\]
is a zero-sum subsequence of $B$ with
\(
| T | = | L | + | J  | \ge |A|.
\)
If $ |A'|\notin J$, then
\[
S_1 = L f \prod_{i \in [1,|A'|-1]\setminus J } e_{i}
\]
is a non-empty zero-sum subsequence of $B$ with
\[
| S_1 | = |L| + 1+ |A'| - 1 - |J| \le    1 + (r - |A'|) + |A'|  - |J| = r - |J|+1.
\]
Since $ | J  | \ge |A| - | L | \ge |A| - ( 1 + (r - |A'|))=|B|  - r-1$, we get
$|S_1|\le 2r+2-|B|$.
Yet, since by assumption $|B| \ge (3r +6)/2$, this implies $|S_1| \le |B|-r$, a contradiction to the minimality of $|A|$.

If $ |A'|\in J$, then
\[
S_2 = L f e_{|A'|}\prod_{i \in [1,|A'|]\setminus J } e_{i}
\]
is a non-empty zero-sum subsequence of $B$ with
\[
| S_2 | = |L| + 2 + |A'|  - |J| \le    1 + (r - |A'|) + 2 + |A'| - |J| = r + 3 - |J|.
\]
As above, we get $|S_2|\le 2r+4-|B|$, which also implies $|S_2| \le |B|-r$, again a contradiction.

\noindent\textbf{Case 2:} Suppose that $|A|=|B| - r$. As above, we may then write
\[
B = \left( f \prod_{i=1}^{r-1} e_{i} \right) \left( e_{r} \prod_{j=1}^{|A|-1} f_j  \right).
\]
with $e_1, \dots, e_r$ independent.
Since $\sigma ( \prod_{j=1}^{|A|-1} f_j  )= e_r$, there exists some $f_j$, say $f_1$, such that for
 $J \subset [1,r]$ with $f_1= \sum_{j \in J} e_j$ we have $r \in J$.

Then
\(
T = f_{1}  \prod_{j \in J } e_j
\)
is a non-empty zero-sum subsequence of $B$ with
\(
| T | = 1 + | J  | \ge |A|
\)
and
\[
S =  f_{1} f e_r\prod_{i \in [1,r]\setminus J} e_{i}
\]
is also a non-empty zero-sum subsequence of $B$ with
\[
| S | = 3 + r - |J| \le r + 4 - |A| = r + 4 - (|B| - r)   \le |B| - r - 1,
\]
the last inequality since
\(
|B| \ge  (3r+6)/2,
\)
Again, this contradicts the minimality of $|A|$.

\noindent\textbf{Case 3:} Suppose that $|A| = |B| - r - 1$.  We may then write
\[
B = \left( f \prod_{i=1}^{r} e_{i} \right) \left( \prod_{j=1}^{|A|} f_j  \right)
\]
with $e_1, \dots, e_r$ independent. Again, let  $J \subset [1,r]$ with $f_1= \sum_{j \in J} e_j$,
Now we have that
\(
T = f_{1}   \prod_{j \in J } e_j
\)
is a non-empty zero-sum subsequence of $|B|$ with
\(
| T | = 1 + | J  | \ge |A|
\)
and
\(
S =  f_{1} f  \prod_{i \in [1,r]\setminus J} e_{i}
\)
is also a zero-sum subsequence of $B$ with
\[
| S |=r+2-|J| \le  r + 3 - |A| = r + 3 - (|B| - r - 1) \le |B| - r - 2,
\]
the last inequality  since
\(
|B| \ge (3r+6)/2,
\)
a contradiction.
\end{proof}

Combining this upper bound with the lower bound established in Theorem \ref{elb_thm}, we obtain the precise value of $\mathsf{D}_2(C_2^r)$ for $r=4$ and $r=6$, namely $8$ and $11$, resp. For $r=2$ and $r=3$ the bounds also yield the exact value (cf.~Remark \ref{as_rem}).
For $r=5$, the lower and upper bounds do not coincide, they are $9$ and $10$, resp., yet below we will show that equality holds at the upper bound.
For larger $r$ our bounds are far apart, yet for sufficiently large $r$, better bounds can be obtained using results from 
coding theory, namely $1.26 r \le \mathsf{D}_2(C_2^r) \le 1.40 r$ (cf.~\cite{cohenzemor} and see \cite{plagneschmid} for recent additional investigations in this direction).

We turn to the investigation of $\mathsf{D}_k(C_2^r)$ for larger $k$. We start by investigating $\mathsf{D}_k(C_2^4)$ and $\mathsf{D}_k(C_2^5)$. We determine the exact value for each $k$. For $C_2^4$, the result can be found partially and implicitly in the work of P.C.~Baayen \cite{baayen}. 

\begin{theorem}
\label{e2g_thm_4}
$\mathsf{D}_0(C_2^4)=5$ and $k_{\mathsf{D}}(C_2^4)=3$. Additionally, $\mathsf{D}_{1}(C_2^4)=5$ and $\mathsf{D}_{2}(C_2^4)=8$.
\end{theorem}
\begin{proof}
For $k=1$ the statement is well-known and for $k=2$ see the remark
after Theorem \ref{e2g_thm_ub2}.
For $k=3$, we note that by Proposition \ref{ub_prop_bas}.3 with $\ell=3$ and Theorem \ref{e2g_thm_l3ex}, we have $\mathsf{D}_{3}(C_2^4)\le \max \{ 8 + 3, 9 - 1 \}=11$.
It remains to show that $\mathsf{D}_{3}(C_2^4)\ge 11$.
If there exists some squarefree $B\in \mathcal{B}(C_2^4)$ with $0\nmid B$ and $|B|=11$, then by Lemma \ref{e2g_lem_lb}
$\max \mathsf{L}(B)\le 11/3 < 4$, that is $B \in \mathcal{M}_3(C_2^4)$, establishing the claim.
Yet, since there exists some squarefree $C \in \mathcal{B}(C_2^4)$ with $0\mid C$ and $|C|=5$, the squarefree sequence with support $C_2^4\setminus \supp(C)$ has this property.

Now, for $k \ge 4$, we get by Lemma \ref{rlb_lem_lb2} $\mathsf{D}_k(C_2^4)\ge 11 + (k-3)2$.
It thus remains to show that $\mathsf{D}_{k}(C_2^4) \le 5 + 2k$ for $k \ge 4$.
First, we consider $k=4$. Again, by Proposition \ref{ub_prop_bas}.3 we get $\mathsf{D}_{k}(C_2^4) \le 11 + 3$.
Suppose, there exists some $B \in \mathcal{M}_4(C_2^4)$ with $|B|=14$.
We note that each squarefree sequence over $C_2^4$ of length $14$ is not a zero-sum sequence, since its sum is equal to the sum of the two elements that it does not contain.
Thus, $B$ has a zero-sum subsequence of length at most $2$, by Proposition \ref{ub_prop_bas}.3, a contradiction to $\mathsf{D}_{3}(C_2^4)=11$.
Consequently, $\mathsf{D}_4(C_2^4)=13$.
For $k \ge 5$, the claim follows by Proposition \ref{ub_prop_bas}.3, since $\mathsf{s}_{\le 2}(C_2^4)=\eta(C_2^4)=16$.
\end{proof}

\begin{theorem}
\label{e2g_thm_5}
$\mathsf{D}_0(C_2^5)=11$ and $k_{\mathsf{D}}(C_2^5)=10$. Additionally, $\mathsf{D}_{1}(C_2^5)=6$, $\mathsf{D}_{2}(C_2^5)=10$,
$\mathsf{D}_{3}(C_2^5)=13$, $\mathsf{D}_{4}(C_2^5)=16$, $\mathsf{D}_{5}(C_2^5)=19$, $\mathsf{D}_{6}(C_2^5)=21$, $\mathsf{D}_{7}(C_2^5)=23$,
$\mathsf{D}_{8}(C_2^5)=26$, and $\mathsf{D}_{9}(C_2^5)=28$.
\end{theorem}
We note that the sequence $(\mathsf{D}_{k+1}(C_2^5) - \mathsf{D}_{k}(C_2^5))_{k \in \mathbb{N}}$ is not non-increasing.
Moreover, we point out that just to determine $\mathsf{D}_0(C_2^5)$ and $k_{\mathsf{D}}(C_2^5)$, would be considerably less effort than to determine all values of $\mathsf{D}_k(C_2^5)$; this fact is blurred by the way we prove this result (to do it differently would add a lot of redundancy), yet in Remark \ref{e2g_rem_thm5} we discuss this in more detail.

For clarity of exposition, we break up the proof of Theorem \ref{e2g_thm_5} into auxiliary results.
In the following result we determine the structure of sequences of maximal length in $\mathcal{A}_{k}(C_2^5)$ for $k \in [2, 5]$, and in particular determine $\mathsf{D}_k(C_2^5)$ for $k \in [2,5]$. The information on the structure of these sequences is used to determine $\mathsf{D}_6(C_2^5)$.

\begin{proposition}
\label{e2g_prop_small}
Let $B \in \mathcal{B}(C_2^5)$ and let $k \in [2,5]$.
Then, $B\in \mathcal{A}_k(C_2^5)$ with $|B|= \mathsf{D}_{k}(C_2^5)$ if and only if $|B|=4+3k$, $B$ is squarefree, $0\nmid B$, and there exist $g_1\dots g_{k-2}\mid B$ such that $\supp((g_1\dots g_{k-2})^{-1}B)$ is contained in the non-zero coset of a subgroup of index $2$.
\end{proposition}
An alternate, and perhaps more natural, formulation for the last condition for $k=5$ is that $\supp(B)$ contains the non-zero coset of a subgroup of index $2$;
yet, to highlight the analogy to the preceding statements we chose the other formulation. Moreover, we point out that for $k=4$ the elements $g_1,g_2$ actually can be contained in the coset as well; for $k=3$ and $k=5$ this is clearly impossible.
Moreover, note that for $k=2$ the condition $0\nmid B$ is redundant, since $\supp(B)$ is contained in a non-zero coset.

\begin{proof}
We consider each   $k \in [2,5]$ separately; yet, the results for larger $k$ build on those for smaller ones.

\noindent
\textbf{Case $k=2$:}
First, suppose that $|B|=10$, $B$ is squarefree, and $\supp(B)$ is contained in the non-zero coset of a subgroup of index $2$. Then, $B$ has no non-empty zero-sum subsequence of length less than $4$. Thus $\max \mathsf{L}(B)\le 2$ and it follows that $B\in \mathcal{A}_2(C_2^5)$.
Now, suppose $B\in \mathcal{A}_2(C_2^5)$ with $|B|= \mathsf{D}_2(C_2^5)$.
We observe that, since by Lemma \ref{e2g_lem_D2m} $\mathsf{s}_{\le 4}(C_2^5)\le 9$ and applying Proposition \ref{ub_prop_bas}.3 $\mathsf{D}_2(C_2^5)\le 10$, and by the above example equality holds.
Again, by Proposition \ref{ub_prop_bas}.3, $B$ has no non-empty zero-sum subsequence of length less than $4$. Thus, by Lemma \ref{e2g_lem_l3}, $B$ is squarefree and $\supp(B)$ sum-free.
More precisely, by Theorem \ref{e2g_thm_l3ex} it follows that $\supp(B)$ is contained in the non-zero coset of a subgroup of index $2$; note that $B$ cannot  have the other form mention in that result, since  then it would not be a zero-sum sequences.

\noindent
\textbf{Case $k=3$:} First, suppose $|B|=13$, $B$ is squarefree, $0\nmid B$ and there exists some $g\mid B$ such that $\supp(g^{-1}B)$ is contained in the non-zero coset of a subgroup of index $2$. Let $B=A_1\dots A_{\ell}$ with   $A_i\in \mathcal{A}(C_2^5)$; assume that $g\mid A_{\ell}$. Since $B$ is squarefree and $0 \nmid B$, we have $|A_i| \ge 3$ for each $i \in [1, \ell]$. Moreover, we get that $|A_i|$ is even for $i \in [1, \ell -1]$. Thus, $\ell \le 3$ and in fact $\ell = 3$ by the result for $k=2$.

Now, suppose $B\in \mathcal{A}_3(C_2^5)$ with $|B|= \mathsf{D}_3(C_2^5)$. By the argument above we get that $|B|=\mathsf{D}_3(C_2^5)\ge 13$.
First, suppose $|B|> 13$. By Proposition \ref{ub_prop_bas}.3 and $\mathsf{D}_2(C_2^5)=10$, we get that  $B$ has no non-empty zero-sum subsequence of length less than $4$, and thus as above $B$ is squarefree and by Theorem \ref{e2g_thm_l3ex} $\supp(B)$ is contained in the non-zero coset of a subgroup of index $2$.
Since over $C_2^4$ there exists no squarefree zero-sum sequence of length $14$ (cf.~Proof of Theorem \ref{e2g_thm_4}), we get that $|B| \neq 14$. Yet, $|B| = 16$ cannot hold either, since $B$ would have a zero-sum subsequence of length $4$, contradicting $\mathsf{D}_2(C_2^5)=10$.
Thus, we have $|B|=13$. As above we get that $0 \nmid B$ and that $B$ is squarefree. Since $\supp(B)$ is not contained in the non-zero coset of a subgroup of index $2$, it follows that there exists some $A\in \mathcal{A}(C_2^5)$ with $|A|=3$ and $A\mid B$. Let $C=A^{-1}B$; we know that $C \in \mathcal{A}_2(C_2^5)$.
By the result for $k=2$ we get that  $\supp(C)$ is  contained in the non-zero coset of a subgroup of index $2$, say $e_1 + G_1$.  Let $g \mid A$ such that $g\notin G_1$; clearly such an element exists.
By Theorem \ref{e2g_thm_l3ex} there exists some $A'\in \mathcal{A}(C_2^5)$ with $|A'|=3$ and   $A'\mid gC$.  Let $C'=A'^{-1}B$. As above we get that  $\supp(C')$ is  contained in the non-zero coset of a subgroup of index $2$, say, $e_2 + G_2$. If $e_1+G_1=e_2+G_2$, then this coset contains all elements of $\supp(g^{-1}B)$ and we are done. Thus, suppose $e_1+G_1\neq e_2+G_2$. Then, $(e_1 + G_1) \cap (e_2+G_2)=e_0+G_0$ where $G_0$ is a group of rank $3$.
Clearly $\supp(\gcd(C,C'))\subset e_0+G_0$.  Since $|\gcd(C,C')|\ge 8$, we get that indeed $\supp(\gcd(C,C'))= e_0+G_0$.
This implies that $\sigma(\gcd(C,C'))=0$. However, this is impossible, since if this were the case we would get that $\gcd(C,C')^{-1}C$ is a zero-sum subsequence of length $2$ of $B$, contradicting the fact that $B$ is squarefree. So, $e_1+G_1=e_2+G_2$ and the claim is established.

\noindent
\textbf{Case $k=4$:} First, suppose $|B|=16$, $B$ is squarefree, $0\nmid B$, and there exist $gh\mid B$ such that $\supp((gh)^{-1}B)$ is contained in the non-zero coset of a subgroup of index $2$. Similarly as in the argument for $k=3$, let $B=A_1\dots A_{\ell}$  and assume that $gh\mid A_{\ell-1}A_{\ell}$. We have $|A_i| \ge 3$ for each $i \in [1, \ell]$ and $|A_i|$ is even for $i \in [1, \ell -2]$. Thus, $\ell \le 4$ and in fact $\ell = 4$.

Now, suppose $B\in \mathcal{A}_4(C_2^5)$ with $|B|= \mathsf{D}_4(C_2^5)$. By the argument above we get that $|B|=\mathsf{D}_4(C_2^5)\ge 16$.
By Theorem \ref{e2g_thm_l3ex} we have $\mathsf{s}_{\le 3}(C_2^5)=17$ and by Proposition \ref{ub_prop_bas}.3 it thus follows that indeed $\mathsf{D}_4(C_2^5)= 16$.
If $\supp(B)$ is contained in the non-zero coset of a subgroup of index $2$, we are done. If this is not the case, we get by Theorem \ref{e2g_thm_l3ex} that there exists some $A\in \mathcal{A}(C_2^5)$ with $|A|=3$ and $A\mid B$; let $C=A^{-1}B$. We get that $C \in \mathcal{A}_3(C_2^5)$ and $|C|=13$. By the result for $k=3$ we know the structure of $C$; let $g\mid C$ such that $\supp(g^{-1}C)$ is contained in $e+G'$ the non-zero coset of $G'$, a subgroup of index $2$.
If $\supp(A)\cap (e+G')\neq \emptyset$, then $|\supp(A)\cap (e+G')|= 2$ and we are done.
Thus, we may assume that $\supp(A)\cap (e+G')= \emptyset$.

By the argument for $k=3$, we know that there exist some  $A'\in \mathcal{A}(C_2^5)$ with $|A'|=3$ and $A'\mid C$, and for $C'=A'^{-1}C$ we have that $\supp(C')$ is contained in the non-zero coset of a subgroup of index $2$. More precisely, we know that $g \mid A'$ and that $\supp(C')\subset e+ G'$; note that $|\supp(C')|=10$ and thus the non-zero coset of a subgroup of index $2$ that contains $\supp(C')$ is uniquely determined.

We consider $AC'$. We assert that $\max \mathsf{L}(AC')\ge 4$, which implies $\max\mathsf{L}(B)\ge 5$, a contradiction.
We have $|AC'|=13$. Since $\supp(A)\cap (e+G')=\emptyset$, it follows that no non-zero coset of a subgroup of index $2$ contains $12$ elements of $\supp(AC')$; for $e+G'$ this is clear, and any other non-zero coset of a subgroup of index $2$ can contain at most $8$ elements of $\supp(C')$.
Thus, by the result for $k=3$ we get that $AC'\notin \mathcal{A}_3(C_2^5)$ and the claim follows.

\noindent
\textbf{Case $k=5$:}  The first part of the argument is analogous to the one for $k=3$ and $k=4$, in this case we have at most $3$ zero-sum sequences of length $3$ and all other sequences have length at least $4$.
Moreover, since $\mathsf{s}_{\le 3}(C_2^5)=17$ it follows by Proposition \ref{ub_prop_bas}.3, and the example that $\mathsf{D}_5(C_2^5)=19$.

Suppose $B\in \mathcal{A}_5(C_2^5)$ with $|B|= \mathsf{D}_5(C_2^5)$. We get that $B=AC$ with $A$ a zero-sum sequence of length $3$ and $C \in \mathcal{A}_4(C_2^5)$.
By the result for $k=4$ we know that all except at most $2$ elements of $\supp(C)$ are contained in the non-zero coset $e+G'$ of $G'$, a subgroup of index $2$.
If $\supp(A)\cap (e+G')\neq \emptyset$, then this intersection contains $2$ elements, and we are done.
Also, if  $\supp(C)= e + G'$, we are done.
Thus, suppose neither is the case.

We get, by the argument for $k=4$, that $C=A_1A_2C'$ with zero-sum sequence $A_i$ of length $3$ and
$C'\in \mathcal{A}_2(C_2^5)$ with $\supp(C')\subset e+G'$.
We consider $AC'$. Since by assumption $\supp(A)\cap e+ G'=\emptyset$, this sequence is not of the form given by the result for $k=4$.
Thus, $\max \mathsf{L}(AC')\ge 4$, and $\max \mathsf{L}(B)\ge 2+\max \mathsf{L}(AC')\ge 6$, a contradiction.
\end{proof}

\begin{lemma}
\label{e2g_lem_67}
$\mathsf{D}_6(C_2^5)=21$ and $\mathsf{D}_7(C_2^5)=23$.
\end{lemma}
\begin{proof}
First, we consider $\mathsf{D}_6(C_2^5)$.
By Proposition \ref{ub_prop_bas}, Theorem \ref{e2g_thm_l3ex} and Proposition \ref{e2g_prop_small} we know that $\mathsf{D}_6(C_2^5)\le \mathsf{D}_5(C_2^5)+3=22$, and by Lemma \ref{rlb_lem_lb2} we get that
$\mathsf{D}_6(C_2^5)\ge \mathsf{D}_5(C_2^5)+2=21$.
Thus, suppose $B \in \mathcal{B}(C_2^5)$ with $|B|=22$. We have to show that $B\notin \mathcal{M}_6(C_2^5)$, i.e., $\max \mathsf{L}(B)>6$.
By the results mentioned above, it is clear that $B$ is squarefree and $0 \nmid B$.
By Theorem \ref{e2g_thm_l3ex} we know that there exists some $A \in \mathcal{A}(C_2^5)$ with $|A|=3$ such that $A\mid B$.
Let $B=AC$. If $C \notin \mathcal{M}_5(C_2^5)$, we are done.
Thus assume that $C \in \mathcal{M}_5(C_2^5)$. By Proposition \ref{e2g_prop_small} we get that there exists some non-zero coset of a subgroup of index $2$, say $e+G'$, and $ghf\mid C$ such that $\supp((ghf)^{-1}C)= e+G'$.
We note that $\supp(A) \subset G'$.
Moreover, we note that, by repeated application of Theorem \ref{e2g_thm_l3ex}, there exist $A_g,A_h, A_f \in \mathcal{A}(C_2^5)$ each of length $3$, where $g,h,f$ is contained in the respective minimal zero-sum sequence,  such that $A_gA_hA_f\mid C$.
Let $C=A_gA_hA_fD$. We note that $\supp(D) \subset e+G'$ and we consider $AD$. By Proposition \ref{e2g_prop_small} we get that $AD \notin \mathcal{M}_3(C_2^5)$.
Yet, this implies that $B=A_gA_hA_fAD\notin \mathcal{M}_6(C_2^5)$.

Now, we consider $\mathsf{D}_7(C_2^5)$. Analogously as above, we see that it suffices to consider a sequence $B \in \mathcal{B}(C_2^5)$ with $|B|=24$ and to show that
$B \notin \mathcal{M}_7(C_2^5)$. Again, we may assume that $B$ is squarefree and $0 \nmid B$.
Let $B'$ denote the squarefree sequence with support $C_2^5\setminus (\supp(B)\cup \{0\})$.
Since $B'$ is a zero-sum sequences of length $7$ it follows that $B'=A_3A_4$ with $A_i \in \mathcal{A}(C_2^5)$ of length $3$ and $4$, resp.
We recall that by Proposition \ref{e2g_prop_maxfull} there exists a factorization $\zeta$ of $BB'$ with $|\zeta|=10$.
Yet, we need the stronger assertion that there exists a factorization $\zeta^{\ast}$ of $BB'$ with $|\zeta^{\ast}|=10$ and $A_3A_4 \mid \zeta^{\ast}$.

Let $s$ denote the rank of $\langle \supp(B')\rangle$; clearly $ 3\le s \le 5$. Moreover, let $G' = \langle \supp(A_4) \rangle$.
For $s=3$ and $s=5$ a factorization that is generated by the proof of Proposition \ref{e2g_prop_maxfull} essentially has the required property.
For clarity, we make this more explicit.
If $s=3$, then $\supp(B')= G' \setminus \{0\}$ and $\supp(B)= C_2^5 \setminus G'$.
Thus, it follows (cf.~the proof of Proposition \ref{e2g_prop_maxfull}) that $\max L(B)= |G'|=8$ and the claim is established.
If $s=5$, let $A_3= e_1e_2(e_1+e_2)$ and we note that $C_2^5= G' \oplus \langle e_1,e_2 \rangle$.
The proof of Proposition \ref{e2g_prop_small}, with respect to this decomposition of the groups and using the bijection $\varphi$ given in the proof of Lemma \ref{e2g_lem_full} for ``$r=3$'', yields a factorization $\zeta=\zeta'\zeta''$ where $\supp(\pi(z'))=G' \setminus \{0\}$ and $\supp(\pi(\zeta''))=C_2^{5}\setminus G'$. And, we have  $A_3 \mid \zeta''$ and we may assume that $A_4 \mid \zeta'$.

Now, suppose $s=4$. Let $A_4=f_1f_2f_3(f_1+f_2+f_3)$ with independent elements $f_i$.
We observe that $|\supp(A_3)\cap G'|=1$, and since this element is non-zero and not contained in $\supp(A_4)$, we may assume that it is equal to $f_1+f_2$. Let $e_1\in \supp(A_3)\setminus \{f_1+f_2\}$. Then $A_3=e_1(f_1+f_2)(e_1+f_1+f_2)$.
Let $e_2 \in C_2^5$ such that $C_2^5 = G' \oplus \langle e_1, e_2 \rangle$. As above, let $\zeta=\zeta'\zeta''$ denote the factorization of $BB'$ of length $10$ that is given by the proof Proposition \ref{e2g_prop_maxfull}, and as noted above we may assume that $A_4 \mid \zeta'$. Note that this implies that $A_4^{-1}\zeta'= (f_1+f_2)(f_1+f_3)(f_2+f_3)$.
Moreover, we have $e_1e_2(e_1+e_2)\mid \zeta''$ and $(e_1 + f_1 +f_2) (e_2 + f_1 +f_3)(e_1+e_2 + f_1 +f_2)\mid \zeta''$; again, we use the bijection $\varphi$ as defined in the proof of Lemma \ref{e2g_lem_full}.

Now, we construct a new factorization $\zeta^{\ast}$ of $B'B$. Let
$\xi_1=(e_1e_2(e_1+e_2)) \cdot ((e_1 + f_1 +f_2) (e_2 + f_1 +f_3)(e_1+e_2 + f_1 +f_2))\cdot ((f_1+f_2)(f_1+f_3)(f_2+f_3))$
and let $\xi_2 = A_3 \cdot(e_2(e_2+f_1+f_3)(f_1+f_3))\cdot (e_3(e_3+f_2+f_3)(f_2+f_3))$, then set
\[\zeta^{\ast}=   \xi_1^{-1}  \zeta \xi_2 .\] This is indeed a factorization of $B'B$ of length $10$ and is divisible by $A_3$ and $A_4$.
\end{proof}

Having the preparatory result at hand, we complete the proof of Theorem \ref{e2g_thm_5}.

\begin{proof}[Proof  of Theorem \ref{e2g_thm_5}]
For $k \le 7$, the result were established in Proposition \ref{e2g_prop_small} and Lemma \ref{e2g_lem_67}.
By Proposition \ref{ub_prop_bas}.3 and Theorem \ref{e2g_thm_l3ex}, we know that $\mathsf{D}_8(C_2^5)\le \mathsf{D}_7(C_2^5) + 3 = 26$.
We observe that there exists a squarefree $B \in \mathcal{B}(C_2^5)$ with $0 \nmid B$ and $|B| = 26$. By Lemma \ref{e2g_lem_lb}, we have $\max\mathsf{L}(B)\le 26/3$, and thus $B \in \mathcal{M}_8(C_2^5)$. Thus, $\mathsf{D}_8(C_2^5)= 26$.
As above, it follows that $\mathsf{D}_9(C_2^5)\le \mathsf{D}_8(C_2^5) + 3 = 29$.
Suppose there exists some $B \in \mathcal{M}_9(C_2^5)$ with $|B|=29$. We observe that $0\mid B$ or $B$ is not squarefree; in any case $B$ has a non-empty zero-sum subsequence of length at most $2$. Yet, by Proposition \ref{ub_prop_bas}.3 this contradicts $\mathsf{D}_8(C_2^5) = 26$.
Thus, we get $\mathsf{D}_9(C_2^5)\le 28$ and by Lemma \ref{rlb_lem_lb2} equality holds.
As above, we get $\mathsf{D}_{10}(C_2^5)\le  31$ and the existence of a squarefree $B \in \mathcal{B}(C_2^5)$ with $0\nmid B$ and $|B|=31$ shows that equality holds.
Finally, since $\mathsf{s}_{\le 2}(C_2^5)=\eta(C_2^5)=32$, we get by Proposition \ref{ub_prop_bas}.3 that $\mathsf{D}_{10+\ell}(C_2^5)\le 31 +2\ell$ for each $\ell \in \mathbb{N}$, and conversely by Lemma \ref{rlb_lem_lb2} that  $\mathsf{D}_{10+\ell}(C_2^5)\ge 31 +2\ell$, completing the argument.
\end{proof}

In the following remark, we sketch a different argument to show $\mathsf{D}_0(C_2^5)=11$, which does not require to determine all constants $\mathsf{D}_k(C_2^5)$.

\begin{remark}
\label{e2g_rem_thm5}
By Proposition \ref{e2g_prop_tech} we know that $k_{\mathsf{D}}(C_2^5)\le 10$.
Thus, to determine $\mathsf{D}_0(C_2^5)$ it suffices to determine $\mathsf{D}_{10}(C_2^5)$.
By Lemma \ref{e2g_lem_lb} (or Proposition \ref{e2g_prop_maxfull}) we get that  $\mathsf{D}_{10}(C_2^5)\ge 31$.
Suppose that there exists some $B \in \mathcal{M}_{10}(C_2^5)$ such that $|B| \ge 32$; by Lemma \ref{rlb_lem_lb2} we may assume that $0 \nmid B$.
Let $B= B'T^2$ with $B' \in \mathcal{B}(C_2^5)$ squarefree and $T \in \mathcal{F}(C_2^5)$.
We have $10 \ge \max \mathsf{L}(B)\ge \mathsf{L}(B')+|T|$ and $|B'|\ge 32 -2 |T|$.

Now, suppose the following holds; we describe below how these claims can be proved.
\begin{enumerate}
  \item If $|B'|= 28$, then $\max \mathsf{L}(B')\ge 9$.
  \item $\mathsf{D}_k(C_2^5)\le 11 + 2k$ for $k \in [1,7]$.
\end{enumerate}
It is easy to see that $|T|\in [1,9]$. If $|T| \in [3,9]$, we get $|B'| \ge 32 -2|T| > 11 + 2 (10 -|T|) \ge \mathsf{D}_{10-|T|}(C_2^5)$, and thus $\max\mathsf{L}(B')> 10-|T|$, a contradiction.
For $|T|=1$, we note that $|B'|\neq 30$ and if $|B'|=31$, then by Proposition \ref{e2g_prop_maxfull} $\max \mathsf{L}(B')=10$, a contradiction.
And, for $|T|=2$, we note that $|B'|\notin \{29,30\}$ and for $|B'|\in \{28,31\}$, we have by 1. and see above $\max\mathsf{L}(B')\ge 9$, a contradiction.

We explain how to show 1. and 2. For 1.\ the argument is similar to the one used to determine $\mathsf{D}_7(C_2^5)$, though much simpler.
For 2., we first determine $\mathsf{D}_k(C_2^5)$ for $k \in [1,4]$, e.g.,  as in the first parts of Proposition \ref{e2g_prop_small} (the detailed investigation of the structure could be omitted). Then, we use Proposition \ref{ub_prop_bas}.3 and $\mathsf{s}_{\le 3}(C_2^5)=17$, to get $\mathsf{D}_{4 + \ell}(C_2^5)\le \mathsf{D}_4(C_2^5)+3\ell =16+3\ell$.
\end{remark}

We derive bounds for $\mathsf{D}_0 (C_2^r)$ that are asymptotically exact. We exclude $r=1$, since this case is well-known (see Remark \ref{as_rem}) and would have to be considered separately.

\begin{theorem}
Let $r \in \mathbb{N}\setminus \{1\}$. Then
\[\left \lceil \frac{2^r - 1}{3} \right \rceil \le \mathsf{D}_0(C_2^r)\le \left \lceil \frac{2^r - 1}{3}\right \rceil +  2^{r/2}.\]
\end{theorem}
We note that for $r \in [2,5]$, by Remark \ref{as_rem} and Theorems \ref{e2g_thm_4} and \ref{e2g_thm_5} equality holds at the lower bound.
It is not clear to us whether this is to be expected for all $r$.
We only mention that the proof actually yields a slightly better upper bound for $\mathsf{D}_0(C_2^r)$, namely $(3+ 2^{(r+1)/2})/2$ and there seems to be room for further improvements.
Finally, we point out that direct application of Proposition \ref{ub_prop_lengthlb} with $\overline{\ell}=(2,3,4)$ would yield $(5/12) 2^r + \mathcal{O}(2^{r/2})$ as an upper bound.

\begin{proof}
Let $k_0 = \lfloor (2^r-1)/3 \rfloor$.
Let $B\in \mathcal{B}(C_2^r)$ the squarefree sequence with support $C_2^r\setminus \{0\}$.
By Lemma \ref{e2g_lem_lb} we have $\max \mathsf{L}(B)\le   |B|/3  $, and thus $B \in \mathcal{M}_{k_0}(C_2^r)$.
Thus, $\mathsf{D}_{k_0}(C_2^r)\ge 2^r - 1$ and $\mathsf{D}_0(C_2^r)\ge 2^r - 1 - 2 k_0 = \lceil (2^r-1)/3 \rceil$, establishing the lower bound.

Let $k_1= k_{\mathsf{D}}(C_2^r)$. Let $B \in \mathcal{A}_{k_1}(C_2^r)$ with $|B| = \mathsf{D}_{k_1}(C_2^r)$. By Proposition \ref{e2g_prop_tech} we know that $B$ is squarefree and that $0\nmid B$.
Since $|B|= \mathsf{D}_0(C_2^r) + 2 \max \mathsf{L}(B)$, an upper bound for $|B| - 2 \max\mathsf{L}(B)$ is an upper bound for $\mathsf{D}_0(C_2^r)$; we proceed to establish such an upper bound.

Let $C\in \mathcal{B}(C_2^r)$ be the squarefree sequence with support $C_2^r \setminus (\supp(B) \cup  \{0\})$.
By Proposition \ref{e2g_prop_maxfull} we know that $\max \mathsf{L}(BC) = \lfloor (2^r - 1)/3 \rfloor$; let $\zeta \in \mathsf{Z}(BC)$ be a factorization of maximal length. We note that the factorization $\zeta$ consist of minimal zero-sum sequences of length $3$ and possibly one minimal zero-sum sequence of length $4$.

Furthermore, let $\zeta= \zeta'\zeta''$ where $\zeta''$ is minimal with $C \mid \pi(\zeta'')$, in other words $\zeta''$ consists of those minimal zero-sum sequences containing an element of $C$.
Let $B'= \pi(\zeta')$ and $B'' = B'^{-1}B$.
We have $\max \mathsf{L}(B)\ge \max \mathsf{L}(B') + \max \mathsf{L}(B'')$.
Since $\zeta'\mid \zeta$ and the remark on the structure of $\zeta$ above, it follows that
 $\max \mathsf{L}(B')= \lfloor |B'|/3 \rfloor$.
By Remark \ref{ub_rem} and Lemma \ref{e2g_lem_D2m} we know that
\[\max \mathsf{L}(B'')\ge (|B''|- \mathsf{s}_{\le 4}(C_2^r) + 1)/4 \ge (|B''|-  2^{1/2} 2^{r/2})/4.
\]

Consequently,
\[
\max \mathsf{L}(B)\ge \lfloor |B'|/3 \rfloor + (|B''|-  2^{1/2} 2^{r/2})/4 \ge |B|/4 + |B'|/12 - (2/3 + 2^{1/2} 2^{r/2}/4).\]

Next, we establish a lower bound for $|B'|$.
By definition of $\zeta''$, we have $|\zeta''|\le |C|$. Since $\zeta''\mid \zeta$, it follows similarly as above that
$|\pi(\zeta'')| \le 3 |\zeta''| + 1 \le 3 |C| +1$. Thus,
$|B''| \le 2|C| + 1 = 2(2^r - 1 - |B|) +1$ and $|B'| = |B| - |B''|\ge |B|-  2(2^r - 1 - |B|) -1 = 3 |B|  -  2^{r+1} +1 $.
Combining these results we get
\[
\begin{split}
\max \mathsf{L}(B) & \ge   |B|/4 +  ( 3 |B|  -  2^{r+1} +1     )  /12 - (2/3 + 2^{1/2} 2^{r/2}/4)\\
& =|B|/2  -2^{r+1}/12 - (7/12 + 2^{1/2} 2^{r/2}/4).
\end{split}
\]
Therefore,
\[
\begin{split}
|B|- 2\max \mathsf{L}(B) & \le |B| -  2 (|B|/2  - 2^{r+1}/12 - (7/12 + 2^{1/2} 2^{r/2}/4))\\
&  = 2^r/3 + 7/6+ 2^{1/2} 2^{r/2}/2.
\end{split}
\]
For $r\ge 5$ this establishes the upper bound, and for $r \in [2,4]$ the precise value of $\mathsf{D}_0(C_2^r)$ is known by Remark \ref{as_rem} and Theorem \ref{e2g_thm_4}.
\end{proof}

We end our investigation by establishing a variant of Theorem \ref{ub_thm_ind} that is optimized for elementary $2$-groups.

\begin{theorem}
Let $k,r \in \mathbb{N}$ and $s \in [0, r]$. Then
\[\mathsf{D}_k(C_{2}^{r}) \le   \mathsf{D}_{\mathsf{D}_{k}(C_2^s) - s} (C_2^{r-s})   + s.\]
\end{theorem}
\begin{proof}
Let $\ell = \mathsf{D}_k(C_2^s) - s$.
Let $B \in \mathcal{A}_k(G)$ with $|B|= \mathsf{D}_k(C_2^r)$ and suppose $|B|$ exceeds the claimed upper bound.
By \eqref{elb_eq_1} we know that $\langle \supp(B)\rangle = C_2^r$; let $\{e_1, \dots, e_r \} \subset \supp(B)$ be a basis of $C_2^r$.
Let $G' = \langle e_1, \dots, e_s \rangle$ and let $\varphi: C_2^r \to C_2^r / G'$ denote the canonical map; we have $C_2^r/ G' \cong C_2^{r-s}$.
We note that $\varphi(B)=0^sT$ with $T = \varphi((\prod_{i=1}^s e_i)^{-1}B)$.
Since $T \in \mathcal{B}(C_2^{r-s})$ and $|T| > \mathsf{D}_{\ell} (C_2^{r-s})$, we get that $T= \prod_{i=1}^{\ell}T_i$ with non-empty zero-sum sequences $T_i$ over $C_2^r / G'$. Let $(\prod_{i=1}^s e_i)^{-1}B = \prod_{i=1}^{\ell}S_i$ such that $\varphi(S_i)=T_i$ for each $i \in [1,\ell]$.

We consider the sequence $B'=( \prod_{i=1}^{s}e_i )(\prod_{j=1}^{\ell} \sigma(S_j))$.
We have $B'\in \mathcal{B}(G')$ and $|B'| > \mathsf{D}_{k}(G')$. Thus $\max \mathsf{L}(B') > k$. However, this contradicts
$B \in \mathcal{A}_k(G)$, since any factorization of $B'$ yields, by replacing $\sigma(S_i)$ by $S_i$, a factorization of $B$ whose length is not smaller.
\end{proof}
We note that one could combine this result, e.g., with the results on $\mathsf{D}_k(C_2^r)$ for $r\le 5$, to establish further explicit upper bounds for $\mathsf{D}_k(C_2^r)$.

\section*{Acknowledgment}
The authors would like to thank 
the referees for corrections and suggestions, D.~Grynkiewicz 
for discussions related to Snevily's conjecture,
and A.~Plagne for information on Sidon sets and for bringing 
the paper of G.~Cohen and G.~Z{\'e}mor to our attention.

%\bibliography{schmid_iu,temp}
%\bibliographystyle{abbrv}

\end{document}